\documentclass{amsart}
\usepackage{mathpazo, url}
\usepackage{microtype}
\usepackage{amssymb, amsxtra, enumerate, xspace, graphicx}
\usepackage[latin1]{inputenc} \usepackage[all]{xy} \SelectTips{cm}{}

\numberwithin{equation}{section}

\numberwithin{subsection}{section}

\allowdisplaybreaks[1]

\newenvironment{enumeratea} {\begin{enumerate}[\upshape(a)]} {\end{enumerate}}

\newenvironment{enumeratei} {\begin{enumerate}[\upshape(i)]} {\end{enumerate}}

\newtheorem*{namedtheorem}{\theoremname} \newcommand{\theoremname}{testing}

\newtheorem{theorem}{Theorem}[section]
\newtheorem{proposition}[theorem]{Proposition}
\newtheorem{proposition-definition}[theorem] {Proposition-Definition}
 \newtheorem{lemma}[theorem]{Lemma}
\newtheorem{conjecture}[theorem]{Conjecture}

\theoremstyle{definition} \newtheorem{definition}[theorem]{Definition}
 \newtheorem{example}[theorem]{Example}
 \newtheorem{remark}[theorem]{Remark}

\theoremstyle{remark} \newtheorem*{claim}{Claim}

 \newcommand\Forms{\operatorname{Forms}}
\newcommand\diag{\operatorname{diag}} \newcommand\Char{\operatorname{char}}
\newcommand\Gr{\operatorname{Gr}}
\newcommand\Hypersurf{\operatorname{Hypersurf}} 
\newcommand{\GL}{\mathrm{GL}} 
\newcommand{\SL}{\mathrm{SL}} \newcommand{\PGL}{\mathrm{PGL}}

 \newcommand\cB{\mathcal{B}}
 
 \newcommand\cF{\mathcal{F}}
\newcommand\cG{\mathcal{G}} 
\newcommand\cI{\mathcal{I}} 
 
\newcommand\cM{\mathcal{M}} \newcommand\cN{\mathcal{N}}
\newcommand\cO{\mathcal{O}}

\newcommand\cU{\mathcal{U}} 
 \newcommand\cX{\mathcal{X}}
\newcommand\cY{\mathcal{Y}} 

\renewcommand\AA{\mathbb{A}} 
  
 \newcommand\GG{\mathbb{G}}

 \newcommand\PP{\mathbb{P}}

\newcommand\bM{\mathbf{M}}  
  
  \newcommand\bU{\mathbf{U}}
  \newcommand\bX{\mathbf{X}}
\newcommand\bY{\mathbf{Y}}

\newcommand\fM{\mathfrak{M}} \renewcommand\frm{\mathfrak{m}}

\newcommand\arr{\ifinner\to\else\longrightarrow\fi}

\newcommand\arrto{\ifinner\mapsto\else\longmapsto\fi}

\renewcommand\H{\operatorname{H}}

\newcommand\eqdef{\overset{\mathrm{\scriptscriptstyle def}} =}

\newcommand\into{\hookrightarrow}

\renewcommand\th{^\text{th}}

\def\displaytimes_#1{\mathrel{\mathop{\times}\limits_{#1}}}

\def\displayotimes_#1{\mathrel{\mathop{\bigotimes}\limits_{#1}}}

\newcommand\spec{\operatorname{Spec}}

\newcommand\cd{\operatorname{cd}}

\newcommand\Spec{\operatorname{Spec}}

\newcommand\codim{\operatorname{codim}}

\newcommand{\proj}{\operatorname{Proj}}

\newcommand\generate[1]{\langle #1 \rangle}

\renewcommand\projlim{\varprojlim}

\newcommand\dash{\nobreakdash-\hspace{0pt}}

\newdir{ >}{{}*!/-5pt/@{>}}

\newcommand\double{\mathbin{\rightrightarrows}}

\newcommand\doublelong[2]{\mathbin{\xymatrix{{}\ar@<3pt>[r]^{#1}
\ar@<-3pt>[r]_{#2}&}}}

\newcommand{\underaut} {\mathop{\underline{\mathrm{Aut}}}\nolimits}

\newlength{\ignora}

\renewcommand{\setminus}{\smallsetminus}

\newcommand{\ind}{\operatorname{ind}}

\newcommand{\gm}{\GG_{\mathrm{m}}}
\newcommand{\ga}[1][\relax]{\GG_{\mathrm{a#1}}}

\newcommand{\ed}{\operatorname{ed}}
\newcommand{\ged}{\operatorname{g\mspace{2mu}ed}}
\newcommand{\field}[1]{(\mathrm{Field/#1})} \newcommand{\catset}{(\mathrm{Set})}

\newcommand{\ds}[1]{[\mspace{-2mu}[#1]\mspace{-2mu}]}
\newcommand{\trdeg}{\operatorname{tr\mspace{2mu}deg}}

\newcommand{\dm}{Deligne--Mumford\xspace}

\newcommand{\sym}{\operatorname{Sym}}

\newcommand{\ur}{\operatorname{R_{u}}}

\newcommand{\orb}{\operatorname{Orb}}

\newcommand{\Orb}{\operatorname{Orb}}

\newcommand{\Stab}{\operatorname{Stab}}

\begin{document}

\title[A genericity theorem and essential dimension of hypersurfaces]{A
genericity theorem for algebraic stacks\\and essential dimension of
hypersurfaces}

\author[Reichstein]{Zinovy Reichstein$^\dagger$}

\author[Vistoli]{Angelo~Vistoli$^\ddagger$}

\address[Reichstein]{Department of Mathematics\\ University of British
Columbia \\ Vancouver, B.C., Canada V6T 1Z2}
\email{reichst@math.ubc.ca}

\address[Vistoli]{Scuola Normale Superiore\\ Piazza dei Cavalieri 7\\ 56126
Pisa\\ Italy}
\email{angelo.vistoli@sns.it}

\begin{abstract} We compute the essential dimension of the functors $\Forms_{n,
d}$ and $\Hypersurf_{n, d}$ of equivalence classes of  homogeneous polynomials
in $n$ variables and hypersurfaces in $\PP^{n-1}$, respectively, over any base
field $k$ of characteristic $0$. Here two polynomials (or hypersurfaces) over
$K$ are considered equivalent if they are related by a linear change of
coordinates with coefficients in $K$. Our proof is based on a new Genericity
Theorem for algebraic stacks, which is of independent interest. As 
another application of the Genericity Theorem, we prove a new result
on the essential dimension of the stack of (not necessarily smooth)
local complete intersection curves.
\end{abstract}



\thanks{$^\dagger$Supported in part by NSERC Discovery and Accelerator
Supplement grants} \thanks{$^\ddagger$Supported in part by the PRIN Project
``Geometria sulle variet\`a algebriche'', financed by MIUR}

\date{March 19, 2011}

\maketitle

\setcounter{tocdepth}{1} \tableofcontents

\section{Introduction}

Let $k$ be a base field of characteristic $0$, $K/k$ be a field extension, and
$F(\underline{x})$ be a homogeneous polynomial (which we call a \emph{form}) of
degree $d$ in the $n$ variables $\underline{x} = (x_1, \dots, x_n)$, with
coefficients in $K$. We say that $F$ \emph{descends to an intermediate field $k
\subset K_0 \subset K$} if there exists a linear change of coordinates $g \in
\GL_{n}(K)$ such that every coefficient of $F(g \cdot \underline{x})$ lies in
$K_0$.

It is natural to look for a ``smallest" subfield $K_0$ to which a given form
$F(\underline{x})$ descends. A minimal such field $K_0$ with respect to
inclusion may not exist, so we ask instead for the minimal transcendence degree
$\trdeg_k K_0$. This number, called the \emph{essential dimension} $\ed_{k}F$ of
$F$, may be thought of as measuring the ``complexity" of $F$. A major goal of
this paper is to compute the maximum of $\ed_{k}F$, taken over all fields $K/k$
and all forms  $F(x_1, \dots, x_n)$ of degree $d$. This integer, usually called
the essential dimension $\ed_{k}\Forms_{n, d}$ of the functor of forms
$\Forms_{n, d}$ depends only on $n$ and $d$; it may be viewed as a measure of
complexity of all forms of degree $d$ in $n$ variables.

We will also be interested in a variant of this problem, where the form
$F(\underline{x}) \in K[x_1, \dots, x_n]$ of degree $d$ is replaced by the
hypersurface \begin{equation} \label{e.hypersurface} H \eqdef \{ (a_1: \dots :
a_n) \mid F(a_1, \dots, a_n)  = 0 \} \end{equation} in $\PP^{n-1}$. Here we say
that $H$ descends to $K_0$ if there exists a linear change of coordinates $g \in
\GL_n(K)$ and a scalar $c \in K^*$ such that every coefficient of $c F(g \cdot
\underline{x})$ lies in $K_0$. Once again, the essential dimension $\ed_k(H)$ of
$H$ is defined as the minimal value of $\trdeg_k K_0$, with the minimum taken
over all fields $K_0/k$ such that $H$ descends to $K_0$. We will be interested
in the essential dimension $\ed_k(\Hypersurf_{n, d})$, defined as the maximal
value of $\ed_k(H)$, where the maximum is taken over all $K/k$ and all forms
$F(\underline{x}) \in K[x_1, \dots, x_n]$ of degree $d$. Here $H$ is the zero
locus of $F$, as in~\eqref{e.hypersurface}.

The study of forms played a central role in 19th century algebra. The problems
of computing $\ed_{k}\Forms_{n, d}$ and $\ed_{k}\Hypersurf_{n, d}$ are quite
natural in this context. However, to the best of our knowledge, these questions
did not appear in the literature prior to the (relatively recent) work of G.
Berhuy and G. Favi, who showed that $\ed_{k}\Hypersurf_{3, 3} = 3$;
see~\cite{bf2}.

In this paper we compute 
$\ed_{k}\Forms_{n, d}$ and $\ed_{k}\Hypersurf_{n, d}$ for all $n, d \ge 1$.
Our main result is as follows.

\begin{theorem} \label{thm.main} Assume that $n \ge 2$ and $d \ge 3$ are
integers and $(n, d) \neq (2, 3)$, $(2, 4)$ or $(3, 3)$.  Then
\begin{enumeratea}
\item $\ed_{k}\Forms_{n, d} = \binom{n + d - 1}{d} - n^2 + 
\cd(\GL_n/\mu_d) + 1$\,.

\item $\ed_{k}\Hypersurf_{n, d} = \binom{n + d - 1}{d} - n^2 +
\cd(\GL_n/\mu_d)$\,.
\end{enumeratea} \end{theorem}

The values of $\ed_{k}\Forms_{n, d}$ and $\ed_{k}\Hypersurf_{n, d}$
for $n, d \ge 1$ not covered by Theorem~\ref{thm.main}
are computed in Section~\ref{sect.small}; the results are
summarized in the following table.   

\begin{center}  
\vspace{0.2cm}
\begin{tabular}{|c| c| c| c|}

$n$ & $d$ & $\ed_{k} \Forms_{n, d}$ & $\ed_{k} \Hypersurf_{n, d}$ \\
\hline       
arbitrary & 1 & $0$ & $0$ \\
\hline       
1 & $\ge 2$ & $1$ & $0$ \\
\hline       
arbitrary & $2$ & $n$ & $n-1$ \\
\hline       
$2$ & $3$ & $2$ & $1$ \\
\hline       
$2$ & $4$ & $3$ & $2$ \\
\hline       
$3$ & $3$ & $4$ & $3$ \\
\hline       
\end{tabular}
\end{center}
\vspace{0.2cm}

The quantity $\cd(\GL_n/\mu_d)$ which appears in the statement of
Theorem~\ref{thm.main} is the canonical dimension of the algebraic 
group $\GL_n/\mu_d$. For the definition and basic 
properties of canonical dimension we refer the reader 
to~Section~\ref{sect.cd}; see also~\cite{ber, km1} for 
a more extensive treatment of this notion. The exact 
value of $\cd(\GL_n/\mu_d)$ is known in the case where $e \eqdef
\gcd(n, d)$ is a prime power $p^j$.  In this case \[  \cd(\GL_n/\mu_d) =
\begin{cases} \text{$p^i - 1$, if $j > 0$}, \\ \text{$0$, otherwise,}
\end{cases} \] where $p^i$ is the highest power of $p$ dividing $n$;
see~\cite[Section 11]{ber}.  More generally, suppose $e = p_1^{j_1} \dots
p_r^{j_r}$ is the prime decomposition of $e$ (with $j_1, \dots, j_r \ge 1$) and
$p_s^{i_s}$ is the highest power of $p_s$ dividing $n$. A conjecture of
J.-L.~Colliot-Th\'el\`ene, N.~A.~Karpenko, and A.~S.~ Merkurjev~\cite[(2)]{ctkm}
implies that \begin{equation} \label{conjecture-ctkm} \cd(\GL_n/\mu_d) = \sum_{s
= 1}^r (p_s^{i_s} - 1). \end{equation} This has only been proved if $e$ is a
prime power (as above) or $n = 6$~\cite[Theorem 1.3]{ctkm}. In these two cases
Theorem~\ref{thm.main} gives the exact value of $\ed_{k}\Forms_{n, d}$ and
$\ed_{k}\Hypersurf_{n, d}$. For other $n$ and $d$ Theorem~\ref{thm.main} reduces
the problems of computing $\ed_{k}\Forms_{n, d}$ and $\ed_{k}\Hypersurf_{n, d}$
to the problem of computing the canonical dimension $\cd(\GL_n/\mu_d)$. For
partial results on the latter problem, see~\cite[Section 11]{ber}.

The notions of essential dimension for forms and hypersurfaces 
are particular cases of Merkurjev's general definition of essential 
dimension of a functor~\cite{bf1}. A special case of this, 
upon which our approach is based, is the essential dimension 
of an algebraic stack. For background material
on this notion we refer the reader to~\cite{brv-jems}.
In particular, $\ed_{k}\Forms_{n, d} = \ed_{k} [A_{n, d}/\GL_n]$ 
and $\ed_{k}\Hypersurf_{n, d} = \ed_{k} [\PP(A_{n, d})/\GL_n]$,
where $A_{n, d}$ is the $\binom{n + d - 1}{d}$-dimensional affine space of forms
of degree $d$ in $n$ variables and $\PP(A_{n, d})$ is the associated $\binom{n +
d - 1}{d}-1$ dimensional projective space of degree $d$ hypersurfaces in
$\PP^{n-1}$.
 (Here, as in the rest of the paper, we will follow the classical
 convention of defining the projectivization $\PP(V)$ of a vector space $V$ over $k$ as
 the projective space of lines in $V$, that is, as $\proj \sym_{k} V^{\vee}$. In the present context, this seems more 
natural than Grothendieck's convention of defining $\PP(V)$ as $\proj \sym_{k} V$.) 
The group $\GL_n$ naturally acts on
these spaces, and $[A_{n, d}/\GL_n]$ and $[\PP(A_{n, d})/\GL_n]$ denote the
quotient stacks for these actions; see~\cite[Example 2.6]{brv-jems}.

The essential dimension of the ``generic hypersurface" of degree $d$ is
$\PP^{n-1}$, i.e., of the hypersurface  $H_{\rm gen}$ cut out
by the ``generic form"
\begin{equation} \label{e.generic-form} 
F_{\rm gen}(x_1, \dots, x_n) = \sum_{i_1 + \dots + i_n = d} 
a_{i_1, \dots, i_n} x_{1}^{i_1} \dots x_{n}^{i_n} = 0 \, ,
\end{equation} 
where $a_{i_1, \dots, i_n}$ are independent variables and $K$ is
the field generated by these variables over $k$, 
was computed in~\cite[Sections 14-15]{ber}. 
The question of computing the essential dimension of the generic
form $F_{\rm gen}$ itself was left open in~\cite{ber}. For $n$ and $d$ as 
in Theorem~\ref{thm.main} we will show that 
$\ed_{k}F_{\rm gen} = \ed_{k}H_{\rm gen} + 1$; 
see Proposition~\ref{prop.generic-ed}.

The key new ingredient in the proof of Theorem~\ref{thm.main} is the following
``Genericity Theorem". Let $\cX$ be a connected algebraic stack with
quasi-affine diagonal that is smooth of finite type over $k$, in which the
automorphism groups are generically finite (for sake of brevity, we say that
$\cX$ is \emph{amenable}). Then we can define the \emph{generic essential
dimension} of $\cX$, denoted by $\ged_{k}\cX$, as the supremum of the essential
dimensions of the \emph{dominant} points $\spec K \arr \cX$. If $\cX$ is
Deligne--Mumford, that is, if all stabilizers are finite, then 
$\ed_{k}\cX = \ged_{k}\cX$; 
see~\cite[Theorem~6.1]{brv-jems}. This result, which we called
the Genericity Theorem for \dm stacks in~\cite{brv-jems}, is not 
sufficient for the applications in the present paper.  Here 
we prove the following stronger theorem conjectured 
in~\cite[Question~6.6]{brv-jems}. 

\begin{theorem}\label{thm:genericity} Let $\cX$ be an amenable stack over $k$.
Let $L$ be a field extension of $k$, and $\xi$ be an object of $\cX(\spec L)$,
such that the automorphism group scheme $\underaut_{L}\xi$ is reductive. Then \[
\ed_{k}\xi \leq \ged_{k} \cX\,. \] In particular, if the automorphism group of
any object of $\cX$ defined over a field is reductive, then $\ed_{k}\cX =
\ged_{k} \cX$. 
\end{theorem}

Note that Theorem~\ref{thm:genericity} fails if the stabilizers are 
not required to be reductive (see~\cite[Example 6.5(b)]{brv-jems}),
even though a weaker statement may be true in this setting 
(see Conjecture~\ref{conjecture.non-reductive}).
We also remark that the locus of points with reductive 
stabilizer is constructible but not necessarily open in $\cX$. 
Thus for the purpose of proving Theorem~\ref{thm:genericity}
it does not suffice to consider the case where all 
stabilizers are reduced.


Theorem~\ref{thm:genericity} implies, in particular, that if the automorphism
group of a form $f(x_1, \dots, x_n)$ is reductive then  $\ed_{k} f \le
\ed_{k} F_{\rm gen}$. To complete the proof of Theorem~\ref{thm.main}(a) we
supplement this inequality with additional computations, carried on in
Section~\ref{sect.non-reductive}, which show that forms $f(x_1, \dots, x_n)$
whose automorphism group is not reductive have low essential dimension; for a
precise statement, see Theorem~\ref{thm:non-reductive}. The proof of
Theorem~\ref{thm.main}(b) is more delicate because the quotient stack
$[\PP(A_{n, d})/\GL_n]$ is not amenable, so the Genericity Theorem cannot be
applied to it directly. We get around this difficulty in
Section~\ref{sect.proof-of-main-thm} by relating $\ed_{k}[\PP(A_{n, d})/\GL_n]$
to the essential dimension of the amenable stack $[\PP(A_{n, d})/\PGL_n]$.

In the last section we use our Genericity Theorem~\ref{thm:genericity}
to prove a new result on the essential dimension of the stack 
of (not necessarily smooth) local complete intersection curves,
strengthening~\cite[Theorem 7.3]{brv-jems}.

\subsection*{Acknowledgments} We are grateful to J.~Alper and 
P.~Brosnan for helpful discussions.

\section{Preliminaries} \label{sect.prel}

\subsection{Special groups} \label{sect.special} A linear algebraic group scheme
$G$ over $k$ is said to be \emph{special} if for every extension $K/k$ we have
$\H^1(K, G) = \{ 1 \}$. Special groups were studied by Serre~\cite[Expos\'e
1]{chevalley2} and classified by Grothendieck~\cite[Expos\'e 5]{chevalley2}
(over an algebraically closed field of characteristic $0$). Note that $G$ is
special if and only if $\ed_{k}G = 0$; see~\cite[Proposition
4.3]{tossici-vistoli}.

The group $\GL_n$ is special by Hilbert's Theorem 90, and so is the special
linear group $\SL_n$. Direct products of special groups are easily seen to be
special. Moreover, in characteristic $0$ the group $G$ is special if and only if
the Levi subgroup of $G$ (which is isomorphic to $G/\ur G$) is special;
see~\cite[Theorem 1.13]{sansuc}. Here $\ur G$ denotes the unipotent radical of
$G$. We record the following fact for future reference.

Let $A$ be a non-zero nilpotent $n \times n$-matrix with entries in $k$ and
$G_A$ be the image of the map $\ga \to \GL_n$ given by $t \to \exp(tA)$. Note
that this map is algebraic, since only finitely many terms in the power series
expansion of $\exp(tA)$ are non-zero.

\begin{lemma} \label{lem.special} \hfil \begin{enumeratea}

\item The centralizer $C$ of $A$ (or equivalently, of $G_A$) in $\GL_n$ is
special.

\item The normalizer $N$ of $G_A$ in $\GL_n$ is special.

\end{enumeratea} \end{lemma}

\begin{proof} (a) By \cite[Propositions 3.10 and 3.8.1]{jantzen} $C$
is a semidirect product $U \rtimes H$, where $U \triangleleft C$ is 
unipotent and $H$ is the direct product of general linear groups
$\GL_r$ for various $r \ge 0$; cf., also~\cite[Section 2]{mcninch}. Thus $H$ =
Levi subgroup of $C$ is special, and part (a) follows.

\smallskip (b) The normalizer $N$ acts on $G_A \simeq \ga$
by conjugation. This gives rise to a homomorphism $\pi \colon N \to \gm =
\underaut_{k} \ga$ whose kernel is the centralizer $C = C_{\GL_n}(A)$. If
$\pi$ is trivial then $N = C$ is special by part (a). If $\pi$ is non-trivial
then it is surjective, and we have an exact sequence \[ 1 \arr C \arr N \arr \gm
\arr 1 \, . \] The long non-abelian cohomology sequence for $\H^{0}$ and
$\H^{1}$ associated with this short exact sequence shows that $\H^1(K, N) = \{ 1
\}$ for every field $K/k$, as desired. \end{proof}

\subsection{Canonical dimension} \label{sect.cd} Let $K$ be a field and $X$ be
either a geometrically integral smooth complete $K$-scheme of finite type or a
$G$-torsor for some connected linear algebraic $K$-group $G$. The canonical
dimension $\cd X$ of $X$ is the minimal value of $\dim Y$, where $Y$ ranges over
all integral closed $K$-subschemes of $X$ admitting a rational 
map $X \dasharrow Y$ defined over $K$. Equivalent definitions 
via generic splitting 
fields and determination functions can be found in~\cite{ber, km1}.

If we fix a base field $k$ and an algebraic $k$-group $G$, the maximal value of
$\cd X$ as $K$ ranges over all field extensions $K/k$ and $X \to \Spec K$ ranges
over all $G_K$-torsors, is denoted by $\cd G$. Moreover, $\cd G = \cd X_{\rm
ver}$, where $X_{\rm ver} \to \Spec K_{\rm ver}$ is a versal $G$-torsor. In
particular, we can construct a versal $G$-torsor by starting with a generically
free linear representation $V$ of $G$ defined over $k$ and setting $K_{\rm ver}
\eqdef k(V)^G$. Then $V$ has a $G$-invariant open subset $U$ which is the total
space of a $G$-torsor $U \to B$, where $k(B) = k(V)^G$. Restricting to the
generic point $\eta$ of $B$, we obtain a versal torsor $X_{\rm ver} \eqdef
U_{\eta} \to \Spec K_{\rm ver}$. For details of this construction we refer the
reader to~\cite[I.5]{gms}.

\begin{lemma} \label{lem.brauer-severi.cd} \hfil \begin{enumeratea}

\item Let $X_1$ and $X_2$ be Brauer-equivalent Brauer--Severi varieties over a
field $K/k$. Then \[ \cd X_1 = \cd X_2 \, . \] In other words, the canonical
dimension $\cd \alpha$ of a Brauer class $\alpha \in \H^{2}(K, \gm)$ is well
defined.

\item Let $G = \GL_n$ or $\SL_n$ and let $C$ be a central subgroup scheme of
$G$. Then for any field $K/k$ and any $(G/C)$-torsor $X \to \Spec K$ we have
$\cd X = \cd \alpha$, where $\alpha$ is the image of the class of $X$ under the
coboundary map $\partial_K \colon \H^1(K, G/C) \to \H^2(K, C) \subseteq
\H^{2}(K, \gm)$ induced by the exact sequence $1 \to C \to G \to G/C \to 1$.

\item Let $K/k$ be a field extension and $\alpha \in \H^2(K, \gm)$ be a Brauer
class of index dividing $n$ and exponent dividing $d$. Then $\cd\alpha \le
\cd(\GL_n/\mu_d)$.

\end{enumeratea} \end{lemma}

\begin{proof} (a) follows from the fact that $X_1$ and $X_2$ have the same
splitting fields $L/K$; see~\cite[Section 10]{ber} or~\cite[Section 2]{km1}.

\smallskip (b) By Hilbert's Theorem 90, $G$ is special, i.e., $\H^1(L, G) = \{ 1
\}$ for any field $L$. Hence, the coboundary map $\partial_L \colon \H^1(L, G/C)
\to \H^2(K, C)$ has trivial kernel for any $L$, and the desired conclusion
follows from~\cite[Lemma 10.2]{ber}.

\smallskip (c) By our assumption, $\alpha$ lies in the image of the coboundary
map \[ \partial_K \colon \H^1(K, \GL_n/\mu_d) \arr \H^2(K, C) \, ; \] cf.,
e.g.,~\cite[Lemma 2.6]{ber}. Part (c) now follows from part (b). \end{proof}

The following result will be used repeatedly in the sequel.

\begin{proposition}  \label{prop.ed-gerbe}\hfil \begin{enumeratea}

\item Let $\cX \to \Spec K$ be a $\gm$-gerbe over a field $K$. Denote the class
of this gerbe in $\H^2(K, \gm)$ by $\alpha$. Then $\ed_K \cX = \cd \alpha$.

\item Let $e \ge 1$ be an integer and $\cX \to \Spec K$ be a $\mu_e$-gerbe over
a field $K$. Denote the class of this gerbe in $\H^2(K, \mu_e)$ by $\beta$. Then
$\ed_K \cX = \cd \beta + 1$.

\end{enumeratea} \end{proposition}

\begin{proof} See~\cite[Theorem 4.1]{brv-jems}. \end{proof}

\subsection{Gerbes and Brauer classes}

Let $\phi \colon \cX \to \overline{\cX}$ be a $\gm$-gerbe over a stack
$\overline{\cX}$. If $L$ is a field and $\xi \in \overline{\cX}(L)$ then,
pulling back $\cX$ to $\Spec L$ we obtain a $\gm$-gerbe $\cX_{\xi}$ over $L$. We
will denote by $\ind(\cX_{\xi})$ and $\exp(\cX_{\xi})$ the index and exponent of
the Brauer class of $\cX_{\xi}$. The following lower semi-continuity properties
of $\ind$ and $\exp$ (as functions of $\xi$) will be used in the proof of
Theorem~\ref{thm.main}(b).

\begin{lemma} \label{lem.index} Let $\phi \colon \cX \to \overline{\cX}$ be a
$\gm$-gerbe over an integral regular algebraic stack $\overline{\cX}$, as above.
Assume further that $\overline{\cX}$ is generically a scheme, with generic point
$\eta \colon \Spec K \to \overline{\cX}$.  Then for any field $L/k$ and any $\xi
\in \overline{\cX}(L)$,

\begin{enumeratea}

\item $\ind(\cX_{\xi})$ divides $\ind(\cX_{\eta})$, and

\item $\exp(\cX_{\xi})$ divides $\exp(\cX_{\eta})$.

\end{enumeratea} \end{lemma}

\begin{proof} (a) The key fact we will use is that if $B$ is a Brauer--Severi
variety over a field $L$ then $\ind(B)$ divides $d$ if and only if $B$ has a
linear subspace of dimension $d-1$ defined over $L$; see~\cite[Proposition
3.4]{artin-bs}.

By \cite[Theorem 6.3]{LMB} there exists a smooth map $T \to X$ such that $\xi$
lifts to a point $\Spec L \to T$; we may assume that $T$ is affine and integral.
The index of the pullback of $\cX_{\eta}$ to the function field $k(T)$ divides
$\ind(\cX_{\eta})$; hence we can substitute $\overline{\cX}$ with $T$, and
assume that $\overline{\cX} = T$ is an affine regular integral variety. The
étale cohomology group $\H^{2}(T, \gm)$ is torsion, because $T$ is regular;
hence, by a well known result of O.~Gabber~\cite{gabber-brauer} 
the class of $\cX$
is represented by a Brauer--Severi scheme $P \to T$.

Let $d$ be the index $\ind(\overline{\cX}_{\eta}) \eqdef \ind(P_{\eta})$ and
$\Gr(P, d-1) \to T$ be the Grassmannian bundle of linear subspaces of dimension
$d-1$ in $P$. The generic fiber $\Gr(P, n-1)_{\eta}$ has a $K$-rational point;
this gives rise to a section $U \to \Gr(P, n-1)$ over some open substack $U$ of
$T$. Let $Y$ be the complement of $U$ in $T$.  If our point $\xi \colon \Spec L
\to \overline{\cX}$ lands in $U$, then the pullback $P_{\xi}$ has a linear
subspace of dimension $d-1$ defined over $L$, and we are done. Thus we may
assume that $\xi \in Y(L)$. The morphism $\xi\colon \Spec L \to T$ extends to a
morphism $\Spec R \to T$, where $R$ is a DVR with residue field $L$, such that
the generic point of $\Spec R$ lands in the complement of $Y$ in $T$. The
pullback $\Gr(P, d-1)_R$ of $\Gr(P, d-1)$ to $\Spec R$ then has a section over
the generic point. By the valuative criterion of properness this section extends
to a section $\Spec R \to \Gr(P, d-1)$. Specializing to the closed
point of $\Spec R$, we obtain a desired section $\Spec L \to \Gr(P, d-1)$. This
shows that $P_{\xi}$ has degree dividing $d$, as claimed.

\smallskip (b) Set $e \eqdef \exp(\cX_{\xi})$ and apply part (a) to the
$e\th$ power $\cY$ of the gerbe $\cX$. Since $\cY_{\eta}$ is trivial (i.e.,
has index $1$), so is $\cY_{\xi}$. 
But $\cY_{\xi}$ is the $e\th$ power of the class of $\cX_{\xi}$,
and we are done.

An alternative proof of part (b) is based on the fact 
that a Brauer-Severi variety $B \to \Spec L$ over a
field $L$ has index dividing $e$ if and only if $P$ contains a hypersurface of
degree $e$ defined over $L$; see~\cite[(5.2)]{artin-bs}. We may thus
proceed exactly as in the proof of part~(a), with the same $T$ 
and $P \arr T$, but using the Hilbert scheme $H(P, e) \to T$ of 
hypersurfaces of degree $e$ in $P$ instead of the Grassmannian. 
\end{proof}

\section{Amenable stacks and generic essential dimension}
\label{sect.generic-ed}

\begin{definition} Let $\cX$ be an algebraic stack over $k$. We say that $\cX$
is \emph{amenable} if the following conditions hold.

\begin{enumeratea}

\item $\cX$ is integral with quasi-affine diagonal.

\item $\cX$ is locally of finite type and smooth over $k$.

\item There exists a non-empty open substack of $\cX$ that is a \dm stack.

\end{enumeratea} \end{definition}

Any irreducible algebraic stack has a \emph{generic gerbe}, the residual gerbe
at any dominant point $\spec K \arr \cX$ \cite[\S~11]{LMB}. For
amenable stacks, there is an alternate description. Let $\cX$ be an amenable
stack over $k$, and $\cU$ a non-empty open substack which is \dm. After
shrinking $\cU$, we may assume that the inertia stack $\cI_{\cU}$ is finite over
$\cU$. Let $\bU$ be the moduli space of $\cU$, whose existence is proved
in~\cite{keel-mori}, and let $k(\bX)$ be 
its residue field. The generic gerbe $\cX_{k(\bX)} \arr \spec k(\bX)$ 
is then the fiber product $\spec k(\bX) \times_{\bU} \cU$. 
The dimension $\dim \cX$ is the dimension of $\cU$, or,
equivalently, the dimension of $\bU$.

\begin{example} \label{ex.amenable2} Consider the action of a linear algebraic
group defined over $k$ on a smooth integral $k$-scheme $X$, locally of finite
type. Then the quotient stack $[X/G]$ is amenable if and only if the stabilizer
$\Stab_G(x)$ of a general point $x \in X$ is finite.

Of particular interest to us will be the $\GL_n$-actions on $A_{n, d}$, the
$\binom{n + d - 1}{d}$\dash dimensional affine space of forms of degree $d$ in
$n$ variables, and $\PP(A_{n, d})$ = $\binom{n + d - 1}{d}-1$ dimensional
projective space of degree $d$ hypersurfaces in $\PP^{n-1}$, as well as the
$\PGL_n$-action on $\PP(A_{n, d})$.

Since the center of $\GL_n$ acts trivially on $\PP(A_{n, d})$, the stack
$[\PP(A_{n, d})/\GL_n]$ is not amenable.  On the other hand, it is classically
known that the stabilizer of any smooth hypersurface in $\PP^{n-1}$ of degree $d
\ge 3$ is finite; see, e.g.,~\cite[Theorem~2.1]{orlik-solomon}
or~\cite{matsumura-monsky}. From this we deduce that the stacks $[\PP(A_{n,
d})/\PGL_n]$ and $[A_{n, d}/\GL_n]$ are both amenable for any $n \ge 2$ and $d
\ge 3$.

Moreover, if $n \ge 2$, $d \ge 3$ and $(n, d) \neq (2, 3)$, $(2, 4)$ or $(3, 3)$
then the stabilizer of a general hypersurface in $\PP^{n-1}$ of degree $d$ is
trivial; see~\cite{matsumura-monsky}. For these values of $n$ and $d$ the 
quotient stack $[\PP(A_{n, d})/\PGL_n]$ is generically a scheme of dimension \[
\dim \PP(A_{n, d}) - \dim\PGL_n = \binom{n + d - 1}{d} - n^2 \, . \]
\end{example}

\begin{definition} \label{def.generic-ed} The \emph{generic essential dimension}
of an amenable stack $\cX$ is \[ \ged_{k} \cX \eqdef \ed_{k(\bX)}\cX_{k(\bX)} +
\dim \cX\,. \] \end{definition}

Alternatively,  $\ged_{k} \cX$ is the supremum of the essential 
dimension of $\zeta \in \cX(K)$, taken over all field extensions
$K/k$ and all dominant $\zeta \colon \spec K \arr \cX$. By the Genericity 
Theorem for \dm stacks~\cite[Theorem~6.1]{brv-jems}, we
see that $\ged_{k}\cX$ is the essential dimension 
of any open substack of $\cX$ that is a \dm stack.

We will now compute the generic essential dimension of the quotient stacks
$[A_{n, d}/\GL_n]$ and $[\PP(A_{n, d})/\GL_n]$ for $n$ and $d$ as in the
statement of Theorem~\ref{thm.main}. Recall that $\ged_{k}[A_{n, d}/\GL_n] =
\ed_{k}F_{\rm gen}$ and $\ged_{k}[\PP(A_{n, d})/\GL_n] = \ed_{k}H_{\rm gen}$,
where $F_{\rm gen}$ is the generic forms of degree $d$ in $n$ variables and
$H_{\rm gen}$ is the generic hypersurfaces, as in~\eqref{e.generic-form}.

\begin{proposition} \label{prop.generic-ed} Let $n \ge 2$ and $d \ge 3$ be
integers.  Assume further that $(n, d) \neq (2, 3)$, $(2, 4)$ or $(3, 3)$.  Then

\begin{enumeratea}

\item $\ged_{k}[\PP(A_{n, d})/\GL_n] = \binom{n + d - 1}{d} - n^2 +
\cd(\GL_n/\mu_d)$.

\item $\ged_{k}[A_{n, d}/\GL_n] = \binom{n + d - 1}{d} - n^2 +
\cd(\GL_n/\mu_d) + 1$.

\end{enumeratea} \end{proposition}

Part (a) was previously known; see~\cite[Theorem 15.1]{ber}. Part (b) answers an
open question from~\cite[Remark 14.8]{ber}.

\begin{proof} Let $\cX = [\PP(A_{n, d}) /\GL_n]$, $\cY = [A_{n, d} /\GL_n]$, and
$\overline{\cX} = [\PP(A_{n, d}) /\PGL_n]$. Consider the diagram 
\[ \xymatrix{  & \cY_{\eta}  \ar[dl] \ar[ddl] & \cY \ar[dl] 
\ar[ddl]^{\text{\tiny $\mu_d$-gerbe}}  \\ 
\cX_{\eta} \ar[d] & \cX \ar[d]_{\text{\tiny $\gm$-gerbe}} & 
\\ \eta \ar[r] & \overline{\cX} & } \] For $n$ and $d$ as in the statement of
the proposition, $[\PP(A_{n, d}) /\PGL_n]$ is generically a scheme (see
Example~\ref{ex.amenable2}). Denote the generic point of this scheme by $\eta$
and its function field by $k(\eta)$. The pull-backs $\cY_{k(\eta)}$ and
$\cX_{k(\eta)}$ are, respectively, a $\mu_d$-gerbe and a $\gm$-gerbe over
$k(\eta)$; these two gerbes give rise to the same class $\alpha \in
\H^2(k(\eta), \mu_d) \subset \H^2(k(\eta), \gm)$. By
Proposition~\ref{prop.ed-gerbe} 
\[\ed_{k(\eta)} \cY_{k(\eta)} = \cd \alpha \quad
\text{and} \quad \ed_{k(\eta)}\cX_{k(\eta)} = \cd \alpha + 1. \] Since 
\[ \trdeg_k k(\eta) = \binom{n + d - 1}{d} - n^2 \, , \] it remains 
to show that
\begin{equation} \label{e.versal-class} \cd \alpha = \cd(\GL_n/\mu_d) \, .
\end{equation} The action of $G = \GL_n/\mu_d$ on $A_{n, d}$ is linear and
generically free. Thus it gives rise to a versal $G$-torsor $t \in \H^1(k(\eta),
G)$, and $\alpha$ is the image of $t$ under the natural coboundary map
$\H^1(k(\eta), G) \to \H^2(k(\eta), \mu_d)$ associated with 
the exact sequence $1 \to \mu_d \to \GL_n \to G \to 1$. 
As we explained in Section~\ref{sect.cd},
$\cd t = \cd(\GL_n/\mu_d)$. On the other hand, by
Lemma~\ref{lem.brauer-severi.cd}(b), $\cd \alpha = \cd t$,
and~\eqref{e.versal-class} follows. 
\end{proof}

\section{Gerbe-like stacks} 
The purpose of the next two sections is to prove the
Genericity Theorem~\ref{thm:genericity}. The proof of the genericity theorem for
\dm stacks in~\cite{brv-jems} relied on a stronger form of genericity for
gerbes; see~\cite[Theorem 5.13]{brv-jems}. Our proof of
Theorem~\ref{thm:genericity} will follow a similar pattern, except that instead
of working with gerbes we will need to work in the more general setting of
gerbe-like stacks, defined below. The main result of this section,
Theorem~\ref{thm:genericity-gerbes}, is a strong form of genericity for
gerbe-like stacks.

\begin{definition} A \dm stack $\cX$ is \emph{gerbe-like} if its inertia stack
$\cI_{\cX}$ is étale over $\cX$.

If $\cX$ is an algebraic stack, the \emph{gerbe-like part} $\cX^{0}$ of $\cX$ is
the largest open substack of $\cX$ that is \dm and gerbe-like. \end{definition}

\begin{remark} If an algebraic stack $\cX$ is \dm, then the inertia stack
$\cI_{\cX} \arr \cX$ is unramified. Hence, if $\cX$ is also reduced then by
generic flatness the gerbe-like part $\cX^{0}$ of $\cX$ is dense in $\cX$.
\end{remark}

\begin{lemma}\label{lemma1} Let $\cX$ be a reduced \dm stack. Suppose that the
inertia stack $\cI_{\cX}$ is finite and étale over $\cX$. Then $\cX$ is a proper
étale gerbe over an algebraic space. \end{lemma}

\begin{remark} The condition that $\cX$ be reduced can be eliminated. However,
it makes the proof marginally simpler, and will be satisfied in all cases of
interest to us in this paper. \end{remark}

\begin{proof} Let $\bX$ be the moduli space of $\cX$; we claim that $\cX$ is a
proper étale gerbe over $\bX$. This is a local problem in the étale topology of
$\bX$. Hence, after passing to an étale covering of $\bX$, we may assume 
that $\bX$ is a connected scheme, and there exists a finite reduced connected
scheme $U$, with a finite group $G$ acting on $U$, such that $\cX = [U/G]$. The
pullback of $\cI_{\cX}$ to $U$ is the closed  subscheme of $G \times U$ defined
as representing the functor of pairs $(g, u)$ with $gu= u$. The fact that
this pullback is étale over $U$ translates into the condition 
that the order of the stabilizer of a geometric point is locally 
constant on $U$. Since $U$ is connected, this means that 
there exists a subgroup $H$ of $G$ that is the stabilizer of all 
the geometric points of $G$; this subgroup is necessarily
normal. The induced action of $(G/H)$ is free, and $U/(G/H) = \bX$; hence $U$ is
étale over $\bX$, and $\cX = [U/G]$ is a gerbe banded by $H$ over $\bX$.
\end{proof}

\begin{lemma}\label{lemma2} Suppose that $\cX$ is a gerbe-like \dm stack, $\cY
\arr \cX$ a representable unramified morphism. Then $\cY$ is also gerbe-like.
\end{lemma}

\begin{proof} The inertia stack $\cI_{\cX}$ of a stack $\cX$ is the fiber
product $\cX \times_{\cX \times \cX} \cX$. We have a diagram \[ \xymatrix@=15pt{
\cY\times_{\cX}\cY\times_{\cX} \cI_{\cX} \ar[rrr]\ar[rd]\ar[ddd]&&& \cY
\times_{\cX}\cY\ar[ld]\ar[ddd]\\ & \cI_{\cX}\ar[d]\ar[r]&\cX \ar[d]\\
& \cX\ar[r]&\cX \times \cX\\ \cY\times_{\cX}\cY \ar[ru]\ar[rrr]&&& \cY \times
\cY \ar[lu] } \] in which all the squares are cartesian. This implies the
equality \[ \cY\times_{\cX}\cY\times_{\cX} \cI_{\cX} = (\cY \times_{\cX}\cY)
\times_{\cX} (\cY \times_{\cX}\cY)\,, \] which in turn tells us that $(\cY
\times_{\cX}\cY) \times_{\cX} (\cY \times_{\cX}\cY)$ is étale over $\cY
\times_{\cX}\cY$. The hypotheses on $\cY \arr \cX$ imply that the diagonal $\cY
\arr \cY \times_{\cX} \cY$ is an open embedding. Thus $\cI_{\cY} =
\cY_{\cY\times\cY}$ is an open substack of $(\cY \times_{\cX}\cY) \times_{\cX}
(\cY \times_{\cX}\cY)$, so it is étale over $\cY \times_{\cX} \cY$ hence it is
étale over $\cY$, as claimed. \end{proof}

From this and the results in \cite{brv-jems}, it is easy to deduce the
following. Given a field $L/k$ and $\xi\in \cX(\spec L)$, we denote by
$\codim_{\cX} \xi$ the codimension of the closure of the image of the
corresponding morphism $\spec L \arr \cX$.

\begin{theorem}\label{thm:genericity-gerbes} Let $\cX$ be an integral gerbe-like
\dm stack which is smooth of finite type over a field $k$. Let $L$ be an
extension of $k$ and $\xi \in  \cX(\spec L)$. Then \[ \ed_{\cX} \xi \leq
\ed_{k(\bX)}\cX_{k(\bX)} + \dim \bX - \codim_{\cX} \xi. \] \end{theorem}

\begin{proof} If the inertia stack $\cI_{\cX}$ is finite over $\cX$, then, by
Lemma~\ref{lemma1}, $\cX$ is an étale proper gerbe over a smooth $k$-scheme, and
the statement reduces to \cite[Theorem~5.13]{brv-jems}. In the general case,
from \cite[Lemma~6.4]{brv-jems} we deduce the existence of an étale
representable morphism $\cY \arr \cX$, such that $\cY$ is an integral \dm stack
with finite inertia, and the morphism $\spec L \arr \cY$ factors through $\cY$.
By Lemmas \ref{lemma1} and \ref{lemma2}, the stack $\cY$ is a proper étale gerbe
over a smooth algebraic space, hence \cite[Theorem~5.13]{brv-jems} can be
applied to it. Let $\eta \in \spec \cY$ be a point in $\cY(L)$ mapping to $\xi$.
Then we have the relations \begin{align*} \ed_{k}\xi &\leq \ed_{k}\eta,\\
\ed_{k(\cX)}\cX_{k(\bX)} &\geq \ed_{k(\bY)}\text{ and}\\ \codim_{\cX}\xi &=
\codim_{\cY}\eta; \end{align*} hence the general case of the Theorem follows
from \cite[Theorem~5.13]{brv-jems}. \end{proof}

\section{The genericity theorem} \label{sect.genericity}

We now proceed with the proof of Theorem~\ref{thm:genericity}. As before, let
$\cY$ be the closure of the image of $\xi\colon \spec L \arr \cX$. The stack
$\cY$ is integral, and since $\Char k = 0$, $\cY$ is generically smooth. Let
$\pi\colon \cM \arr \AA^{1}_{k}$ the deformation to the normal bundle of $\cY$
inside $\cX$; then $\pi^{-1}(\AA_{k}^{1} \setminus \{0\}) = \cX\times_{\spec
k}(\AA_{k}^{1} \setminus \{0\})$, while $\pi^{-1}(0)$ is isomorphic to the
normal bundle $\cN$ of $\cY$ in $\cX$.

\begin{lemma}\label{lemma3}  $\cM^{0} \cap \cN \neq \emptyset\,.$ \end{lemma}

Theorem~\ref{thm:genericity} follows from Lemma~\ref{lemma3} and
Theorem~\ref{thm:genericity-gerbes} by the same argument as
in~\cite[Theorem~6.1]{brv-jems}. This argument is quite short, and we reproduce
it here for sake of completeness.

Let $L$ be an extension of $k$ and let $\xi$ be an object of $\cX(L)$. Call
$\cY$ the closure of the image of the morphism $\xi\colon \spec L \arr \cX$,
with its reduced stack structure. Set $\cN^{0} \eqdef \cM^{0} \cap \cN$. Then
the fiber product $\spec L \times_{\cY} \cN$ is a vector bundle over $\spec L$,
and $\spec L \times_{\cX} \cM^{0}$ is a non-empty open subscheme. Hence
$\xi\colon \spec L \arr \cY$ can be lifted to $\cN^{0}$; this gives an object
$\eta$ of $\cN^{0}(\spec L)$ mapping to $\xi$ in $\cY$. Clearly the essential
dimension of $\xi$ as an object of $\cX$ is the same as its essential dimension
as an object of $\cY$, and $\ed_{k}\xi \leq \ed_{k}\eta$. Let us apply
Theorem~\ref{thm:genericity-gerbes} to the gerbe $\cM^{0}$. The function field
of the moduli space $\bM$ of $\cM$ is $k(\bX)(t)$, and its generic gerbe is
$\cX_{k(\bX)(t)}$; by \cite[Proposition~2.8]{brv-jems}, we have $\ed_{k(\bX)(t)}
\cX_{k(\bX)(t)} \leq \ed_{k(\bX)}\cX_{k(\bX)}$. The composite $\spec L \arr
\cN^{0} \subseteq \cM^{0}$ has codimension at least~$1$, hence we obtain
\begin{align*} \ed_{k} \xi &< \ed_{k(\bX)(t)}\cX_{k(\bX)(t)} + \dim\cM\\ &=
\ed_{k(\bX)}\cX_{k(\bX)} + \dim\cX + 1. \end{align*} This concludes the proof.

\begin{proof}[Proof of Lemma~\ref{lemma3}] If $X$ is a finite dimensional
representation of a group scheme $G$ over a field $k$, we will identify $X$ with
the affine space $\spec (\sym^{\bullet}_{k}X^{\vee})$.

Let us suppose that $X$ is a finite dimensional representation of a linearly
reductive algebraic group $G$ and $Y \subseteq X$ is a subrepresentation. Since
$G$ is reductive, we have a $G$-equivariant splitting $X \simeq Y \oplus Y'$.
Set $\cX \eqdef [X/G]$ and $\cY \eqdef [Y/G]$. Assume that the generic
stabilizer of the action of $G$ on $X$ is finite. Then $\cX$ is amenable, and
$\cY \subseteq \cX$ is a closed integral substack.

It is easy to see that the deformation to the normal bundle $M$ of $Y$ in $X$ is
$G$-equivariantly isomorphic to $X\times_{k} \AA_{k}^{1}$ (where the group $G$
acts trivially on $\AA^{1}_{k}$); the projection \[ Y \times Y'\times_{k}
\AA_{k}^{1} = X \times_{k} \AA_{k}^{1} \simeq M \arr X\times_{k} \AA_{k}^{1} \]
is given by the formula $(y, y', t) \arrto (ty, y', t)$. The deformation to the
normal bundle $\cM$ of $\cY$ is $[M/G] = [X/G] \times_{k} \AA_{k}^{1}$; hence
$\cM^{0} = \cX^{0} \times_{k} \AA_{k}^{1}$, and it is obvious that $\cM^{0} \cap
\cN \neq \emptyset$.

The proof in the general case will be reduced to this by a formal slice
argument.

We may base-change to the algebraic closure of $k$; so we may assume that $k$ is
algebraically closed. By deleting the singular locus of $\cY$, we may assume
that $\cY$ is smooth; by further restricting, we may assume that the inertia
stack $\cI_{\cY}$ is flat over $\cY$, and that all geometric fibers are
reductive, and have the same numbers of connected components.

Let $y_{0} \arr \spec k \arr \cY$ be a general closed point. The residual gerbe
$\cG_{y_{0}} \arr \spec k$ \cite[\S~11]{LMB} admits a section
$\spec k \arr \cG_{y_{0}}$, since $k$ is algebraically closed; hence, if $G
\eqdef \underaut_{k}y_{0}$ it the automorphism group scheme of $y_{0}$, we have
$\cG_{y_{0}} \simeq \cB_{k}G$. The embedding of stacks $\cB_{k}G \arr \cY$ is of
finite type; hence it is easy to see that it is a locally closed embedding, from
Zariski's main Theorem for stacks \cite[Théorème~16.5]{LMB}.

Let $U \arr \cX$ be a smooth morphism, where $U$ is a scheme, together with a
lifting $u_{0}\colon \spec k \arr U$ of $y_{0}$. Call $\cX_{n}$ the $n\th$
infinitesimal neighborhood of $\cB_{k}G$ inside $\cX$: in other words, if $\cU$
is an open substack of $\cX$ containing $\cB_{k}G$ as a closed substack, and we
denote by $I$ the sheaf of ideals of $\cB_{k}G$ inside $\cU$, then $\cX_{n}$ is
the closed substack of $\cX$ defined by the sheaf of ideals $I^{n+1}$. In
particular, $\cX_{0} = \cB_{k}G$.

\begin{lemma}\label{lemma4} There exists a finite dimensional representation $X$
of $G$ with finite generic stabilizers and a trivial subrepresentation $Y
\subseteq X$, with the following property. If we denote by $X_{n}$ and $\cY_{n}$
the $n\th$ infinitesimal neighborhoods of the origin, there is a sequence of
isomorphisms $\cX_{n} \simeq [X_{n}/G]$, compatible with the embeddings $\cX_{n}
\subseteq \cX_{n+1}$ and $[X_{n}/G] \subseteq [X_{n+1}/G]$, that induce
isomorphisms of $\cY_{n}$ with $[Y_{n}/G]$.

Furthermore, denote by $\widehat{X}$ the spectrum of the completion of the local
ring of $X$ at the origin. Then there exists a smooth morphism $U \arr \cX$ with
a closed point $u_{0} \in U$ mapping to $y_{0}$ in $\cY$, and an isomorphism of
$\widehat{X}$ with the spectrum $\widehat{U}$ of the completion of the local
ring of $U$ at $u_{0}$, such that \begin{enumeratea}

\item  the sequence of composites $X_{n} \arr [X_{n}/G] \simeq \cX_{n} \subseteq
\cX$ is obtained by restriction from the composite morphism $\widehat{X} \simeq
\widehat{U} \arr U \arr \cX$, and

\item the inverse image of $Y$ in $\widehat{X}$ corresponds to the inverse image
of $\cY$ in $\widehat{U}$.

\end{enumeratea} \end{lemma}

\begin{proof}

The tautological $G$-torsor $P_{0} \eqdef \spec k \arr \cB_{k}G$ extends to a
$G$-torsor $P_{n} \arr \cX_{n}$, in such a way that the restriction of $P_{n+1}$
to $\cX_{n} \subseteq \cX_{n+1}$ is isomorphic to $P_{n}$, by \cite[Propositions
4.1 and 4.2]{alper-local-quotient}. Each of the stacks $P_{n}$ is in fact a
scheme, because its reduced substack is; in fact, $P_{n}$ must be the spectrum
of a local artinian $k$-algebra $R_{n}$. Clearly, $\cX_{n} = [P_{n}/G]$.

If we denote by $V$ the maximal ideal of $R_{1}$, then $R_{1} = k \oplus V$; the
action of $G$ on $R_{1}$ induces a linear action of $G$ on $V$. The space $V$ is
isomorphic to $I/I^{2}$, which is a coherent sheaf on $\cB_{k}G$, i.e., a
representation of $G$. In turn, $I/I^{2}$ is the cotangent space of deformations
of $\spec k \arr \cX$, that is, the dual to the space of isomorphism classes of
liftings $\spec k[\epsilon] \arr \cX$ of $\spec k \arr \cX$ (here $k[\epsilon]$
denotes, as usual, the ring of dual numbers $k[x]/(x^{2})$).

The homomorphism $R_{n+1} \arr R_{n}$ induced by the embedding $P_{n} \simeq
P_{n+1}|_{\cX_{n}} \subseteq P_{n+1}$ is surjective; its kernel is the ideal
$I^{n}R_{n+1}$. Denote by $R$ the projective limit $\projlim_{n}R_{n}$. (Notice
that, while $G$ acts, by definition, on each of $R_{n}$, this does not,
unfortunately, induced an action of $G$ on $R$, as an algebraic group, unless
$G$ is finite; it if did, this would make the proof conceptually much simpler.)
If $x_{1}$, \dots,~$x_{n}$ is a set of elements of $R$ that project to a basis
for $V$ in $R_{1}$, the ring $R$ is a quotient of the power series ring
$k\ds{x_{1}, \dots,~x_{n}}$ by an ideal $J$ contained in $\frm_{R}^{2}$. We
claim that $J = 0$, i.e., $R$ is a power series ring. For this, it is enough to
check that $R$ is formally smooth over $k$, or, in other words, that if $A$ is a
local artinian $k$-algebra with reside field $k$ and $B$ is a quotient of $A$,
any homomorphism of $k$-algebras $R \arr B$ lifts to a homomorphism $R \arr A$.
Take $n \gg 0$; then $R \arr B$ factors through $R_{n}$. Consider the composite
\[ \spec B \arr \spec R_{n} \arr \cX_{n} \subseteq \cX\,; \] since $\cX$ is
smooth, deformations are unobstructed, i.e., this morphism extends to $\spec A
\arr \cX$. If $n \gg 0$, this factors as $\spec A \arr \cX_{n} \subseteq \cX$;
and since $\spec R_{n}$ is smooth over $\cX_{n}$, as it is a $G$-torsor, the
section $\spec B \arr \spec R_{n}$ of $\spec B \arr \cX_{n}$ lifts to a section
$\spec A \arr R_{n}$, giving the desired extension $R \arr R_{n} \arr A$.

Suppose that $U \arr \cX$ a smooth morphism, where $U$ is a scheme, with a
lifting $u_{0}\colon \spec k \arr U$ of $y_{0}$. Let us assume that $U$ is
minimal at $u_{0}$, or, in other words, that the tangent space of $U$ at $u_{0}$
maps isomorphically onto the deformation space of $\cX$ at $y_{0}$. Since $U$ is
smooth over $\cX$, the morphisms $\spec R_{n} \arr \cX$ lift to a compatible
system of morphism $\spec R_{n} \arr U$, sending $\spec k$ into $u_{0}$; these
yield a morphism $\spec R \arr \cO_{U, u_{0}}$, inducing a homomorphism of
$k$-algebras $\widehat{\cO}_{U, u_{0}} \arr R$. This is a homomorphism of power
series algebras over $k$ which gives an isomorphism of tangent spaces; hence it
is an isomorphism. This shows that the morphisms $\spec R_{n} \arr \cX$ yield a
flat morphism $\spec R \arr \cX$.

Call $d$ the codimension of $\cY$ in $\cX$ at the point $y_{0}$; after a base
change in $R = k\ds{x_{1}, \dots, x_{n}}$, we may assume that the inverse image
of the ideal of $\cY$ in $\cX$ is $(x_{1}, \dots, x_{d})$. Call $X$ the scheme
corresponding to the dual of the vector space $V = \generate{x_{1}, \dots,
x_{n}}$, that is, $X \eqdef \spec \sym^{\bullet}_{k}V$; call $Y$ the linear
subscheme defined by the ideal $(x_{1}, \dots, x_{d})$. Then $\widehat{X} =
\spec R$; the representation $X$ has all the required properties, except that we
have not yet proved that the action of $G$ on $X$ has finite generic
stabilizers, and the representation $Y$ is trivial.

To do this, let us call $I$ the pullback of the inertia stack of $[X/G]$ to $X$;
in other words, $I$ is the subscheme of $G \times_{\spec k} X$ defined by the
equation $gx = x$. We need to show that $I$ is generically finite over $X$.
Since $I$ is a group scheme over $X$, it has equidimensional fibers, hence it is
enough that there is an étale neighborhood $I' \arr I$ of the pair $(1, u_{0})$
in $I$ which is generically finite over $X$. The inverse image of $\cX_{n}$ in
$U$ is the $n\th$ infinitesimal neighborhood $U_{n}$ of $u_{0}$ in $U$. Denote
by $J$ the pullback of the inertia stack of $\cX$ to $U$; we have isomorphisms
$X_{n} \simeq U_{n}$, and compatible isomorphisms of the pullbacks of $I$ and
$J$ to $X_{n}$ and $U_{n}$ respectively. These induce an isomorphism of the
completions of $I$ and $J$ at $(1, x_{0})$ and $(1,u_{0})$ respectively; by
Artin approximation, the morphisms $I \arr X$ and $J \arr U$ are étale-locally
equivalent at $(1, x_{0})$ and $(1,u_{0})$. Since $J$ is generically finite over
$U$ it follows that the action of $G$ on $X$ has generically finite stabilizers.
Also, this implies that the stabilizer of a general closed point of $Y$ is
isomorphic to the isomorphism group scheme of a general point of $\spec k \arr
\cY$, hence it has the same dimension and the same number of connected
components of $G$; hence it equals $G$, and so the action of $G$ on $Y$ is
trivial, as claimed. \end{proof}

Consider the complement $\cF$ of $(\cM \setminus \cN)^{0}$ in $\cM \setminus
\cN$. Since $\cN$ is a divisor on $\cM$, the closure of $\cF$ in $\cM$ will not
contain it; hence, if $v$ is a general $k$-rational point of $\cN$, this has
neighborhood in which the only point of the inertia stack $\cI_{\cM}$ where this
can fail to be étale are over $\cN$; so it is enough to show that $\cI_{\cM}$ is
étale over $\cM$ at a general point of $\cN$. For this it suffices to show that
the locus of point of the inverse image $\cI_{\cN}$ of $\cN$ in $\cI_{\cM}$ at
which $\cI_{\cM}$ is étale over $\cM$ surjects onto $\cY$.

Denote by $N$ the normal bundle of $Y$ in $X$, and by $M$ the deformation to the
normal bundle. If $n_{0}$ is a general closed point of $N$, then the pullback
$I_{M}$ of the inertia stack of $[M/G]$ to $M$ is étale at $n_{0}$. Notice that
$I_{M}$ is étale at a general closed point of the fiber of $N$ over any $y_{0}
\in Y(k)$, since the action of $G$ on $Y$ is trivial, so translation by any
closed point of $Y$ is $G$-equivariant. Denote by $\cM_{n}$ the inverse image of
$\cX_{n} \times_{\spec k}\AA^{1}_{k}$ in $\cM$, and by $M_{n}$ the inverse image
of $X_{n} \times_{\spec k}\AA^{1}_{k}$ in $M$. 

\begin{claim} There is sequence of isomorphisms $\cM_{n} \simeq [M_{n}/G]$ 
compatible with the isomorphisms 
$\cX_{n} \simeq [X_{n}/G]$, and with the identity on $\AA^{1}_{k}$.
\end{claim}

Let us assume the claim. We know that $\cI_{[M/G]}$ is étale over $\cM$ at a
general $k$-rational point of $\cI_{[M/G]}$ lying over the image of the origin
in $[N/G] \subseteq [M/G]$; let $v\in \cI_{\cM}(k)$ be a general $k$-rational
point lying over $\cN$, and let us show that $\cI_{\cM} \arr \cM$ is étale at
$v$. For this we use the infinitesimal criterion for étaleness. Let $A$ be a
finite $k$-algebra with residue field equal to $k$, let $I$ be a proper ideal in
$A$, and consider a commutative diagram \[ \xymatrix{ \spec (A/I) \ar[r]\ar[d]&
\cI_{\cM}\ar[d]\\ \spec A \ar[r]\ar@{-->}[ur] & \cM } \] in which composite
$\spec k \subseteq \spec A \arr \cI_{\cM}$ is isomorphic to $v$; we need to show
that we can fill in the dashed arrow in a unique way. For $n \gg 0$, the
morphism $\spec A \arr \cM$ factor through $\cM_{n}$. Since $\cI_{\cM_{n}} =
\cM_{n} \times_{\cM}\cI_{\cM}$, the square above factors through a square \[
\xymatrix{ \spec (A/I) \ar[r]\ar[d]& \cI_{\cM_{n}}\ar[d]\\ \spec A
\ar[r]\ar@{-->}[ur] & \cM_{n} } \] in which again we have to show the existence
and uniqueness of the lifting. However, the isomorphism $\cM_{n} \simeq
[M_{n}/G]$ induces an isomorphism of the morphism $\cI_{\cM_{n}} \arr \cM_{n}$
with $\cI_{[M_{n}/G]} \arr [M_{n}/G]$; we know that $\cI_{[M_{n}/G]} = [M_{n}/G]
\times_{[M/G]}\cI_{[M/G]}$ is étale at the point corresponding to $v$. Hence the
lifting exists and is unique.

Now let us prove the claim. Set \begin{align*} R &\eqdef U \times_{\cX} U,\\
R_{n} &\eqdef U_{n} \times_{\cX_{n}} U_{n},\\ S &\eqdef X \times_{[X/G]} X = G
\times_{\spec k}X,\\ S_{n} &\eqdef X_{n} \times_{[X_{n}/G]} X_{n} = G
\times_{\spec k}X_{n}\,. \end{align*} The compatible isomorphisms
$\phi_{n}\colon X_{n} \simeq U_{n}$ and $\cX_{n} \simeq [X_{n}/G]$ yield yield
isomorphisms of schemes in groupoids of $R_{n} \double U_{n}$ with $S_{n}
\double X_{n}$, for each $n \geq 0$.

Denote by $I_{U}$ the sheaf of ideals of the inverse image of $\cY \times \{0\}
\subseteq \cX\times \AA^{1}$ in $U\times\AA^{1}$, and by $I_{R}$ the sheaf of
ideals of its inverse image in $R\times\AA^{1}$. Also denote with $J_{U}$ the
sheaf of ideals of $\cY \times \{0\} \subseteq \cX \times \{0\}$ in $U \times
\{0\}$, pushed forward to $U\times\AA^{1}$, and by $J_{R}$ the sheaf of ideals
of its inverse image in $R \times \{0\}$, pushed forward to $R\times\AA^{1}$.
There are natural surjection $I_{U} \twoheadrightarrow J_{U}$ and $I_{R}
\twoheadrightarrow J_{R}$. Set \begin{alignat*}3 U' &\eqdef
\proj_{U\times\AA^{1}} \Bigl(\,\bigoplus_{m = 0}^{\infty}I_{U}^{m}\Bigr), \quad
R' &\eqdef \proj_{R\times\AA^{1}} \Bigl(\,\bigoplus_{m =
0}^{\infty}I_{R}^{m}\Bigr),\\ U'' &\eqdef \proj_{U\times\AA^{1}}
\Bigl(\,\bigoplus_{m = 0}^{\infty}J_{U}^{m}\Bigr), \quad R''\,&\eqdef
\proj_{R\times\AA^{1}} \Bigl(\,\bigoplus_{m = 0}^{\infty}J_{R}^{m}\Bigr).\\
\end{alignat*}

Then $R' \double U'$ is a scheme in groupoids, $R'' \double U''$ is a closed
subgroupoid, and the difference groupoid $R' \setminus R'' \double U' \setminus
U''$ gives a smooth presentation of $\cM$. Let us denote by $U'_{n}$ and
$U''_{n}$ the inverse images of $U_{n}$ in $U'$ and $U''$, and by $R'_{n}$ and
$R''_{n}$ the inverse images of $U_{n}\times U_{n}$ in $R'$ and $R''$. Then the
groupoid $R'_{n} \setminus R''_{n} \double U'_{n} \setminus U''_{n}$ gives a
smooth presentation of $\cM_{n}$; furthermore, we have \begin{align*} U'_{n} &=
\proj_{U_{n}\times\AA^{1}} \Bigl(\,\bigoplus_{m = 0}^{\infty}I_{U}^{m}
\otimes_{\cO_{U\times\AA^{1}}}\cO_{U_{n}\times \AA^{1}} \Bigr)\,,\\ \quad R'_{n}
&= \proj_{R_{n}\times\AA^{1}} \Bigl(\,\bigoplus_{m = 0}^{\infty}I_{R}^{m}
\otimes_{\cO_{U\times\AA^{1}}}\cO_{R_{n}\times \AA^{1}} \Bigr)\,,\\ U''_{n} &=
\proj_{U_{n}\times\AA^{1}} \Bigl(\,\bigoplus_{m = 0}^{\infty}J_{U}^{m}
\otimes_{\cO_{R\times\AA^{1}}}\cO_{U_{n}\times \AA^{1}}\Bigr)\,,\\ \quad
R''_{n}\,&= \proj_{R_{n}\times\AA^{1}} \Bigl(\,\bigoplus_{m =
0}^{\infty}J_{R}^{m} \otimes_{\cO_{R\times\AA^{1}}}\cO_{R_{n}\times
\AA^{1}}\Bigr).\\ \end{align*}

In a completely analogous manner, denote by $I_{X}$ the sheaf of ideals of $Y
\times\{0\}$ in $U\times\AA^{1}$, and by $I_{S}$ the sheaf of ideals of its
inverse image  in $S\times\AA^{1}$. Also denote with $J_{X}$ the sheaf of ideals
of $Y \times \{0\}$ in $X \times \{0\}$, pushed forward to $X\times\AA^{1}$, and
by $J_{S}$ the sheaf of ideals of its inverse image in $S \times \{0\}$, pushed
forward to $S\times\AA^{1}$. Set \begin{align*} X'_{n} &\eqdef
\proj_{X_{n}\times\AA^{1}} \Bigl(\,\bigoplus_{m = 0}^{\infty}I_{X}^{m}
\otimes_{\cO_{X\times\AA^{1}}}\cO_{X_{n}\times \AA^{1}} \Bigr)\,,\\ \quad S'_{n}
&\eqdef \proj_{S_{n}\times\AA^{1}} \Bigl(\,\bigoplus_{m = 0}^{\infty}I_{S}^{m}
\otimes_{\cO_{S\times\AA^{1}}}\cO_{S_{n}\times \AA^{1}} \Bigr)\,,\\ X''_{n}
&\eqdef \proj_{U\times\AA^{1}} \Bigl(\,\bigoplus_{m = 0}^{\infty}J_{U}^{m}
\otimes_{\cO_{R\times\AA^{1}}}\cO_{R_{n}\times \AA^{1}}\Bigr)\,,\\ \quad
S''_{n}\,&\eqdef \proj_{R\times\AA^{1}} \Bigl(\,\bigoplus_{m =
0}^{\infty}J_{R}^{m} \otimes_{\cO_{R\times\AA^{1}}}\cO_{R_{n}\times
\AA^{1}}\Bigr).\\ \end{align*} By the same argument as before, we see that
$[M_{n}/G]$ has a smooth presentation $S'_{n} \setminus S''_{n} \double X'_{n}
\setminus X''_{n}$; hence, to complete the proof we need to establish the
existence of isomorphisms $U'_{n} \simeq X'_{n}$ and $R'_{n} \simeq S'_{n}$,
compatible with the groupoid structures, the isomorphisms $\phi_{n}\colon U_{n}
\simeq X_{n}$ and $\psi_{n}\colon R_{n} \simeq S_{n}$, and the embeddings
$?_{n-1} \subseteq ?_{n}$.

Let us denote by $\widetilde{R}$ and $\widetilde{U}$ the formal schemes obtained
by completing $R\times\AA^{1}$ and $U\times\AA^{1}$ respectively along the
inverse images of $u_{0} \in U$, and by $\widetilde{S}$ and $\widetilde{X}$ the
formal schemes obtained by completing $S\times\AA^{1}$ and $X\times\AA^{1}$
respectively along the inverse images of the origin in $X$. The structure maps
of the schemes in groupoids $R_{n} \double U_{n}$ and $S_{n} \double X_{n}$ pass
to the limit, yielding formal schemes in groupoids $\widetilde{R} \double
\widetilde{U}$ and $\widetilde{S} \double \widetilde{X}$. The isomorphisms
$\phi_{n}\colon U_{n} \simeq X_{n}$ and $\psi_{n}\colon R_{n} \simeq S_{n}$ give
isomorphisms of formal schemes $\widetilde{\phi}\colon \widetilde{U} \simeq
\widetilde{X}$ and $\widetilde{\psi}\colon \widetilde{R} \simeq \widetilde{S}$,
yielding an isomorphism of formal schemes in groupoids of $\widetilde{R} \double
\widetilde{U}$ with $\widetilde{S} \double \widetilde{X}$.

Denote by $I_{\widetilde{U}}$ and $I_{\widetilde{R}}$ the sheaves of ideals of
the inverse images of $\cY\times\{0\} \subseteq\cX \times \AA^{1}$ in
$\widetilde{U}$ and $\widetilde{R}$ respectively, and by $J_{\widetilde{U}}$ and
$J_{\widetilde{R}}$ the pushforwards to $\emph{U}$ and $\widetilde{R}$ of the
sheaves of ideals of the pullbacks of  the inverse images of $\cY\times\{0\}$ in
the inverse images of $\cX \times \{0\}$. Analogously, denote by
$I_{\widetilde{X}}$ and $I_{\widetilde{S}}$ the sheaves of ideals of the inverse
images of $[Y/G]\times\{0\} \subseteq [X/G] \times \AA^{1}$ in $\widetilde{X}$
and $\widetilde{S}$ respectively, and by $J_{\widetilde{X}}$ and
$J_{\widetilde{S}}$ the pushforwards to $\widetilde{X}$ and $\widetilde{S}$ of
the sheaves of ideals of the pullbacks of  the inverse images of
$[Y/G]\times\{0\}$ in the inverse images of $[X/G] \times \{0\}$.

The natural morphisms $\widetilde{u}\colon \widetilde{U} \arr U\times\AA^{1}$,
$\widetilde{r}\colon \widetilde{R} \arr R\times \AA^{1}$, $\widetilde{x}\colon
\widetilde{X} \arr X\times\AA^{1}$ and $\widetilde{s}\colon \widetilde{S} \arr
S\times \AA^{1}$ are flat (this is actually the key point of this part of the
proof). Furthermore, the inverse images of $\cY\times\{0\} \subseteq\cX \times
\AA^{1}$ and of $\cY\times\{0\} \subseteq\cX \times \AA^{1}$ in $\widetilde{U}$
and $\widetilde{R}$ and the inverse images of $[Y/G]\times\{0\} \subseteq [X/G]
\times \AA^{1}$ and of $[X/G]\times\{0\} \subseteq [X/G] \times \AA^{1}$ in
$\widetilde{U}$ and $\widetilde{R}$ correspond under $\widetilde{\phi}$ and
$\widetilde{\psi}$; hence for each $m \geq 0$ we obtain canonical isomorphisms
of coherent sheaves \[ \widetilde{u}^{*}I_{U}^{m} \simeq
\widetilde{\phi}^{*}\widetilde{x}^{*}I_{X}^{m}, \quad\widetilde{r}^{*}I_{R}^{m}
\simeq \widetilde{\psi}^{*}\widetilde{S}^{*}I_{S}^{m}, \] \[
\widetilde{u}^{*}J_{U}^{m} \simeq
\widetilde{\phi}^{*}\widetilde{x}^{*}J_{X}^{m}, \quad\widetilde{r}^{*}J_{R}^{m}
\simeq \widetilde{\psi}^{*}\widetilde{S}^{*}J_{S}^{m}. \] By restricting to
$U_{n}$ and $R_{n}$ we obtain isomorphism of coherent sheaves \begin{align*}
I_{U}^{m}\otimes_{\cO_{U\times\AA^{1}}}\cO_{U_{n}\times \AA^{1}} &\simeq
\phi_{n}^{*}\bigl(I_{X}^{m} \otimes_{\cO_{X\times\AA^{1}}}\cO_{X_{n}\times
\AA^{1}}\bigr),\\ I_{R}^{m}\otimes_{\cO_{R\times\AA^{1}}}\cO_{R_{n}\times
\AA^{1}} &\simeq \phi_{n}^{*}\bigl(I_{S}^{m}
\otimes_{\cO_{X\times\AA^{1}}}\cO_{S_{n}\times \AA^{1}}\bigr),\\
J_{U}^{m}\otimes_{\cO_{U\times\AA^{1}}}\cO_{U_{n}\times \AA^{1}} &\simeq
\phi_{n}^{*}\bigl(J_{X}^{m} \otimes_{\cO_{X\times\AA^{1}}}\cO_{X_{n}\times
\AA^{1}}\bigr),\\ J_{R}^{m}\otimes_{\cO_{R\times\AA^{1}}}\cO_{R_{n}\times
\AA^{1}} &\simeq \phi_{n}^{*}\bigl(J_{S}^{m}
\otimes_{\cO_{X\times\AA^{1}}}\cO_{S_{n}\times \AA^{1}}\bigr). \end{align*} By
summing up over all $m$ we obtain isomorphism of the corresponding Rees
algebras, which yield the desired isomorphisms  $U'_{n} \simeq X'_{n}$ and
$R'_{n} \simeq S'_{n}$.

This ends the proof of Lemma~\ref{lemma3}, and of the Theorem. \end{proof}

On the basis of examples, the following generalization of 
Theorem~\ref{thm:genericity} seems plausible.

\begin{conjecture} \label{conjecture.non-reductive}
Let $\cX$ be an amenable stack over $k$. Let $L$ be an
extension of $k$, and let $\xi$ be an object of $\cX(\spec L)$. 
Then $\ed_{k}\xi \leq \ged_{k} \cX + \dim\ur(\underaut_{L}\xi)$. 
\end{conjecture}

Here $\ur G$ denotes the unipotent radical of $G$, 
as in Section~\ref{sect.special}.  Unfortunately, 
the approach used in this section breaks down in 
the more general setting of the above conjecture:
if the stabilizer is not reductive, the slice theorem 
does not apply.

\section{Essential dimension of $\GL_{n}$-quotients} \label{sect.non-reductive}

Suppose that $G$ is a special affine algebraic group over $k$ acting on a
scheme $X$ locally of finite type over $k$. For each field $L/k$ we have an
equivalence between $[X/G](L)$ and the quotient category for the action of the
discrete group $G(L)$ on the set $X(L)$; hence the essential dimension of
$[X/G]$ equals the essential dimension of the \emph{functor of orbits} \[
\orb_{G,X}\colon \field{k} \arr \catset \] from the category $\field k$ of
extensions of $k$ to the category of sets, sending $L$ to the set of orbits
$\orb_{G,X}(L) \eqdef X(L)/G(L)$; see~\cite[Example 2.6]{brv-jems}.

For the rest of this section we will assume that $X$ is an integral scheme,
locally of finite type and smooth over $k$, and $\GL_{n}$ acts on $X$ with
generically finite stabilizers. Then the quotient stack $[X/\GL_n]$ is amenable;
however the Genericity Theorem~\ref{thm:genericity} does not tell us that
$\ed_{k}[X/\GL_n] = \ged_{k}[X/\GL_n]$ because we are not assuming that the
stabilizer of every point of $X$ is reductive.  Nevertheless, in some cases one
can still establish this equality by estimating $\ed_{k}\xi$ from above and
proving, in an ad-hoc fashion, that $\ed_{k}(\xi) \le \ged_{k}[X/\GL_n]$ for
every $\xi \in [X/\GL_n](L)$ whose automorphism group is not reductive.  The
rest of this section will be devoted to such estimates. These estimates will
ultimately allow us to deduce Theorem~\ref{thm.main} from
Proposition~\ref{prop.generic-ed}.

For each positive integer $\lambda$, denote by $J_{\lambda}$ the $\lambda \times
\lambda$ Jordan block with eigenvalue~$0$, that is, the  $\lambda \times
\lambda$ matrix, that is, the linear transformation $k^{\lambda} \arr
k^{\lambda}$ defined by $e_{1} \arrto 0$ and $e_{i} \arrto e_{i-1}$ for $i = 2,
\dots, \lambda$, where $e_{1}$, \dots,~$e_{\lambda}$ is the canonical basis of
$k^{\lambda}$. Let $\underline{\lambda}$ be a partition of $n$; that is,
$\underline\lambda = (\lambda_{1}, \dots, \lambda_{r})$ is a non-increasing
sequence of positive integers with $\lambda_{1} + \dots + \lambda_{r} = n$. We
denote by $A_{\underline\lambda}$ the $n\times n$ nilpotent matrix which is
written in block form as \[ A_{\underline\lambda} \eqdef \begin{pmatrix}
J_{\lambda_{1}} & 0 & \ldots & 0\\ 0 & J_{\lambda_{2}} & \ldots & 0\\ \vdots &
\vdots & \ddots & \vdots\\ 0 & 0 & \ldots & J_{\lambda_{r}} \end{pmatrix} \]
Every nilpotent $n \times n$ matrix is conjugate to a unique
$A_{\underline\lambda}$. Consider the 1-parameter subgroup
$\omega_{\underline\lambda}\colon \ga \arr \GL_{n}$ defined by
$\omega_{\underline\lambda}(t) = \exp(tA_{\underline\lambda})$. We will usually
assume that $\underline\lambda \neq (1^{n})$; under this assumption
$\omega_{\underline\lambda}$ is injective. Denote by $N_{\underline\lambda}$ the
normalizer of the image of $\omega_{\underline\lambda}$; the group
$N_{\underline\lambda}$ acts on the fixed point locus
$X^{\omega_{\underline\lambda}}$.

\begin{lemma}\label{lem:inequality-a} Let $L$ be a field extension of $k$, and
$\xi$ be an object of $[X/\GL_{n}](L)$ whose automorphism group scheme
$\underaut_{L}\xi$ is not reductive. Then \[ \ed_{k}\xi \le
\max_{\underline\lambda}
\ed_{k}[X^{\omega_{\underline\lambda}}/N_{\underline\lambda}] \, , \] where the
maximum is taken over all partitions $\underline\lambda$ of $n$ different from
$(1^{n})$. \end{lemma}

\begin{proof} Suppose $\xi$ corresponds to the $\GL_n$-orbit of a point $p \in
X(L)$. The automorphism group scheme $\underaut_{L}\xi$ is isomorphic to the
stabilizer $G_{p}$ of $p$ in $\GL_n$. Since we are assuming that this group is
not reductive, $G_{p}$ will contain a copy of $\ga$, which is conjugate to the
image of $\omega_{\underline\lambda}$ for some $\underline\lambda \neq (1^{n})$.
After changing $p$ to a suitable $\GL_n$-translate, we may assume that the image
of $\omega_{\underline\lambda}$ is contained in $G_{p}$; hence $p \in
X^{\omega_{\underline\lambda}}$. The composite $X^{\omega_{\underline\lambda}}
\into X \arr [X/\GL_{n}]$ factors through $[X^{\omega_{\underline\lambda}}/
N_{\underline\lambda}]$; hence $\xi$ is in the essential image of
$[X^{\omega_{\underline\lambda}}/ N_{\underline\lambda}](L)$ in
$[X/\GL_{n}](L)$, and $\ed_{k}\xi \leq
\ed_{k}[X^{\omega_{\underline\lambda}}/N_{\underline\lambda}]$. \end{proof}

\begin{lemma} \label{lem:inequality-b}
$\ed_{k}[X^{\omega_{\underline\lambda}}/N_{\underline\lambda}] \leq  \dim
X^{\omega_{\underline\lambda}}$ for any $\underline\lambda \neq (1^{n})$.
\end{lemma}

\begin{proof} By Lemma~\ref{lem.special}, $N_{\underline\lambda}$ is special.
Hence, \[ \ed_{k}[X^{\omega_{\underline\lambda}}/N_{\underline\lambda}] = \ed_k
\Orb_{N_{\underline\lambda}, X^{\omega_{\underline\lambda}}} \le \dim
X^{\omega_{\underline\lambda}} \, , \] as claimed. \end{proof}

We now further specialize $X$ to the affine space $A_{n, d}$ of forms 
of degree $d$ in the $n$ variables $\underline{x} = (x_1, \dots, x_n)$ 
over $k$. The general linear group $\GL_{n}$ acts 
on $A_{n, d}$ in the usual way, via $(Af)(\underline{x}) \eqdef
f(\underline{x} \cdot A^{-1})$ for any $A \in \GL_n$. 
We are now ready for the main result of this section.

\begin{theorem} \label{thm:non-reductive} Let $L$ be a field extension of $k$,
and $\xi$ be an object of $[A_{n, d}/\GL_{n}](L)$ whose automorphism group
scheme $\underaut_{L}\xi$ is not reductive. Assume that either $d \ge 4$ and $n
\ge 2$ or $d = 3$ and $n \ge 3$. Then $\ed_{k}\xi \leq \binom{n + d - 1}{d} -
n^{2}$. \end{theorem}

\begin{proof} By Lemma~\ref{lem:inequality-a} it suffices to show that
\begin{equation} \label{e.thm:non-reductive}
\ed_{k}[X^{\omega_{\underline\lambda}}/N_{\underline\lambda}] \leq \binom{n + d
- 1}{d} - n^{2} \, . \end{equation} for any $\underline\lambda \neq (1^n)$. The
space $A_{n,d}^{\omega_{\underline\lambda}}$ consists of the forms 
$f(\underline{x})$ such
that \[ f\bigl(\exp(-tA_{\underline{\lambda}} x) \bigr) = 
f(\underline{x}) \, . \] By
differentiating and applying the chain rule, this is equivalent to
 \begin{equation} \label{e.jacobian} 
 \bigtriangledown f(\underline{x}) \cdot A_{\underline\lambda} = 0 \,
 , \end{equation} 
 where $\bigtriangledown f = (\partial f /\partial x_1, \dots, 
 \partial f /\partial x_n)$
is the gradient of $f$. We now proceed
with the proof of~\eqref{e.thm:non-reductive} in three steps.

\smallskip \emph{Case 1:} Assume $d \ge 4$. By Lemma~\ref{lem:inequality-b} it
suffices to show that \begin{equation} \label{e.cases1and2} \dim
A_{n,d}^{\omega_{\underline\lambda}} \leq \binom{n + d - 1}{d} - n^{2} \, .
\end{equation} for any $\underline\lambda \neq (1^n)$. For $\underline\lambda
\neq (1^n)$ formula~\eqref{e.jacobian} tells us that $\partial f/\partial x_1$,
is identically zero. In other words, $f(\underline{x})$ is a form in $x_{2}$,
\dots,~$x_{n}$. Such forms lie in an affine subspace of $A_{n, d}$ isomorphic to
$A_{n-1, d}$. \smallskip Hence, \[ \dim A_{n,d}^{\omega_{\underline\lambda}} \le
\dim A_{n-1, d} = \binom{n + d - 2}{d}  \,, \] and it suffices to prove the
inequality \[ \binom{n+d -1}{d} - \binom{n + d - 2}{d} \geq n^{2} \] or
equivalently, \begin{equation} \label{e.binomial-a} \binom{n + d - 2}{d -1} \geq
n^{2} \, . \end{equation} Since $\binom{n + d - 2}{d -1} = \binom{n + d - 2}{n
-1}$ is an increasing function of $d$ for any given $n \ge 1$, it suffices to
prove~\eqref{e.binomial-a} for $d = 4$. In this case \[ \binom{n + d - 2}{d -1}
- n^{2} = \binom{n+2}{3} - n^2 = \frac{n(n-1)(n-2)}{6} \ge 0 \] for any $n \ge
2$, as desired.

\smallskip \emph{Case 2:} $d = 3$ and $\underline\lambda \neq (1^n)$ or $(2,
1^{n-1})$. Once again, it suffices to prove~\eqref{e.cases1and2}. If 
$\underline\lambda
\neq (1^n)$ or $(2, 1^{n-1})$ then~\eqref{e.jacobian} shows that for every
$f(\underline{x})$ in $A_{n,d}^{\omega_{\underline\lambda}}$ at least two of the partial
derivatives $\partial f/\partial x_{1}$, and $\partial f/\partial x_{i}$ are
identically zero. For notational simplicity we will assume that $i = 2$.  Then
$f(\underline{x})$ is a form in the variables $x_{3}$, \dots,~$x_{n}$.  Hence, \[ \dim
A_{n,3}^{\omega_{\underline\lambda}} \le \dim A_{n-2, 3} = \binom{n}{3} =
\binom{n+2}{3} - n^2 \, , \] as desired.

\smallskip \emph{Case 3:}  Finally assume $d = 3$ and 
$\underline\lambda = (2, 1^{n-1})$. Set $\omega \eqdef
\omega_{(2, 1^{n-1})}$ and $N \eqdef N_{(2, 1^{n-1})}$. By
Lemma~\ref{lem.special}(b) $N$ is a special group. Hence, we may identify
the set of isomorphism classes in $[A^{\omega}_{n,3}/N](K)$ with the set of $N(K)$-orbits in
$[A^{\omega}_{n,3}](K)$, for every field $K/k$.

By~\eqref{e.jacobian} $A_{n,d}^{\omega}$ consists of degree $d$ forms $f(x_1,
\dots, x_n)$ such that $\partial f/\partial x_{1} = 0$. That is, $f(x_1, \dots,
x_n) \in A_{n, d}^{\omega}$ if and only if $f$ is a form in the variables $x_2,
\dots, x_{n}$. Thus $A_{n,d}^{\omega}$ is an affine subspace of  $A_{n,d}$
isomorphic to $A_{n-1, d}$. Our goal is to show that
\begin{equation} \label{e.d=3} \ed_{k} f \le \binom{n+2}{3} - n^{2}
\end{equation} for any $f(\underline{x}) \in A_{n,d}^{\omega}(K)$.

The normalizer $N$ contains a subgroup \[ \Gamma \simeq \GL_{n-2}\ltimes
\ga^{n-2} \,   \] consisting of matrices of the form \[ \begin{pmatrix} 1 & 0 &
0 & \dots & 0 \\ 0 & 1 & 0 & \dots & 0 \\ 0 & a_{3} &  \\ \vdots & \vdots &   &
A &  \\ 0 & a_{n} & \end{pmatrix} \,  \] where $(a_3, \dots, a_n) \in \ga^{n-2}$
and $A \in \GL_{n-2}$. We may assume without loss of generality that the
stabilizer of $f$ in $\Gamma$ does not contain a non-trivial unipotent subgroup.
Indeed, if it does then $f$ is a $\Gamma$-translate of an element of
$A_{n,3}^{\omega_{\underline\lambda}}$ for some $\underline\lambda \neq (1^n)$
or $(2, 1^{n-1})$. For such $f$ the inequality~\eqref{e.d=3} was established in
Case 2.

Since both $\Gamma$ and $N$ are special, it is obvious that
the essential dimension of $f$, viewed as an
element of $[A^{\omega}_{n,3}/N](K) = \Orb_{N, A^{\omega}_{n, 3}}(K)$ is no
greater than the essential dimension of $f$, viewed as an element of
$[A^{\omega}_{n,3}/\Gamma](K) = \Orb_{\Gamma, A^{\omega}_{n, 3}}(K)$. Since we
are assuming that the stabilizer of $f$ in $\Gamma$ does not contain any
non-trivial unipotent subgroups, the Genericity Theorem~\ref{thm:genericity}
tells us that \[ \ed_{k}f \le \ged_{k} [A^{\omega}_{n, 3}/\Gamma] \, . \] By
Lemma~\ref{lem.Gamma} below for $n \ge 4$ the action of $\Gamma \subset
\GL_{n-1}$ on the space $A_{n-1, 3}$ of forms of degree~$3$ in the $n-1$
indeterminates $x_{2}$, \dots,~$x_{n}$ is generically free. Thus
$[A_{n,3}/\Gamma]$ is an amenable stack and is generically a scheme.
Consequently, \[ \ged_{k}  [A^{\omega}_{n, 3}/\Gamma_{n-2}]  = \dim A_{n-1, 3} -
\dim \Gamma_{n-2} = \binom{n+1}{3} - (n-2)^{2} - (n-2)\, . \] A simple
computation shows that \[ \binom{n+1}{3} - (n-2)^{2} - (n-2) \leq \binom{n+2}{3}
- n^{2} \,, \] for any $n \ge 4$; indeed, \[ \binom{n+2}{3} - n^{2} -
\left(\binom{n+1}{3} - (n-2)^{2} - (n-2)\right) = \frac{(n-2)(n-3)}{2} - 1 \ge
0. \qedhere\] \end{proof}

\begin{lemma} \label{lem.Gamma} \hfil \begin{enumeratea}

\item Assume that the base field $k$ is algebraically closed and
$G$ is a connected linear algebraic $k$-group
such that $N \cap Z(G) \ne \{ 1 \}$ for every closed normal 
subgroup $\{ 1 \} \ne N \triangleleft G$. Here $Z(G)$ denotes the center of $G$.
(For example, $G$ could be almost simple or $\GL_n$.) 
Let $H_1$, $H_2$ be closed subgroups of $G$ such that $H_1$ is
finite and $H_2$ contains no non-trivial central elements of $G$. Then for $g
\in G(k)$ in general position, $H_1 \cap g H_2 g^{-1} = \{ 1 \}$.

\item Assume $d \ge 3$ and $n \ge 1$.  
Let $\GL_{n-1}$ be the subgroup of\/ $\GL_n$
acting on the variables $x_2, \dots, x_n$. Then for $f \in A_{n, d}$ 
in general position, $\Stab_{\GL_{n-1}}(f) = \{ 1 \}$.
\end{enumeratea} \end{lemma}

\begin{proof} (a) Assume the contrary. Consider the natural (translation) 
action of $H_1$ on the homogeneous space $G/H_2$. By our assumption 
this action is not generically free. Since $H_1$ is finite, we 
conclude that this action is not
faithful, i.e., some $1 \ne h \in H_1$ acts trivially on $G/H_2$. Then $h$ lies
in $N = \bigcap_{g \in G} g H_2 g^{-1}$. Consequently, $N$ is a non-trivial
normal subgroup of $G$. By our assumption $N$ (and hence, $H_2$) contains a
non-trivial central element of $G$, a contradiction.

\smallskip (b) We may assume that $k$ is algebraically closed. By \cite[Theorem
A]{richardson2}, there exists a subgroup $S_{n, d} \subset \GL_n$ and a dense
open subset $U \subset A_{n, d}$ such that $\Stab_{\GL_n}(f)$ is conjugate to
$S$ for every $f \in U$. Moreover, for $d \ge 3$ (and any $n \ge 1$)
$S_{n, d}$ is a finite group; see Example~\ref{ex.amenable2}.

Write $f(x_1, \dots, x_n) = \sum_{i = 0}^d x_1^{d-i} f_i(x_2, \dots, x_n)$,
where $f_i$ is a form of degree $i$ in $x_2, \dots, x_n$. Clearly $g \in
\GL_{n-1}$ stabilizes $f$ if and only if it stabilizes $f_1, f_2, \dots, f_d$.
In other words, \[ \Stab_{\GL_{n-1}}(f) = \bigcap_{i = 1}^{d} \,
\Stab_{\GL_{n-1}}(f_i) \, . \] Moreover, each $\Stab_{\GL_{n-1}}(f_i)$ is a
conjugate of $S_{n-1, i}$ in $\GL_{n-1}$. Thus it suffices to show that for
$g_1, \dots, g_d$ in general position in $\GL_{n-1}$, \begin{equation}
\label{e.intersection} g_1 S_{n-1, 1} g_1^{-1} \cap \dots \cap g_d S_{n-1, d}
g_d^{-1} = \{ 1 \} \, . \end{equation} This is a consequence of part (a), with
$G = \GL_n$, $H_1 = S_{n-1, 1}$ and $H_2 = S_{n-1, d}$. \end{proof}

\begin{remark} We note that if $d \ge 4$ and $n \ge 3$ or $d = 3$ and $n \ge
5$ then \eqref{e.intersection} is immediate, since $S_{n-1, d} = \{ 1 \}$; 
see Example~\eqref{ex.amenable2}. However, this argument does not 
cover the case where $d = 3$ 
and $n = 3$ or $4$, which are needed for the proof of 
Theorem~\ref{thm:non-reductive} above.
\end{remark}

\section{Proof of Theorem~\ref{thm.main}} \label{sect.proof-of-main-thm}

Theorem~\ref{thm.main}(a) is an immediate consequence of what we have done so
far. Indeed, by Proposition~\ref{prop.generic-ed}(a), \[ \ged_{k} [A_{n,
d}/\GL_n] = \binom{n + d - 1}{d} - n^2 + 1 + \cd(\GL_n/\mu_d) \, . \] Thus it
suffices to show that for any field extension $K/k$ and any $K$-point $\zeta$ of
$[A_{n, d}/\GL_n](K)$, we have \[ \ed_{k}\zeta \le \ged_{k} [A_{n, d}/\GL_n] \,
. \] If the automorphism group scheme $\underaut_{K}(\zeta)$ is reductive, this
is a direct consequence of Theorem~\ref{thm:genericity}, and if
$\underaut_{K}(\zeta)$ is not reductive, then Theorem~\ref{thm:non-reductive}
tells us that \[ \ed_{k}\zeta \le \binom{n + d - 1}{d} - n^2 < \ged_{k} [A_{n,
d}/\GL_n] \, . \] This completes the proof of Theorem~\ref{thm.main}(a). The
rest of this section will be devoted to proving Theorem~\ref{thm.main}(b). The
main complication here is that the stack $[\PP(A_{n, d})/\GL_n]$ is not amenable
(see Example~\ref{ex.amenable2})
and thus our Genericity Theorem~\ref{thm:genericity} does not
apply. We will get around this difficulty by relating $[\PP(A_{n, d})/\GL_n]$ to
the amenable stack $[\PP(A_{n, d})/\PGL_n]$.

\begin{proposition} \label{prop.gerbes} Let $\cX = [\PP(A_{n, d})/\GL_n]$ and
$\overline{\cX} \eqdef [\PP(A_{n, d})/\PGL_n]$, with the natural projection
$\phi \colon \cX \to \overline{\cX}$. Then for any extension $L/k$ and any
$L$-point $\xi \colon \Spec L \to \cX$, \[ \ed_{k}\xi \le \ed_k \phi(\xi)  +
\cd(\GL_n/\mu_d) \, . \] \end{proposition}

\begin{proof} Note that $\cX$ is a gerbe banded by $\gm$ over $\overline{\cX}$.

By the definition of $\ed_k \phi(\xi)$ there exists an intermediate field $k
\subset K \subset L$ such that $\phi(\xi)$ descends to $\Spec K$ and $\trdeg_k K
= \ed_k \phi(\xi)$. Moreover, $\xi \colon \Spec L \to \cX_K$ factors through a
point $\xi_0 \colon \Spec L \to \cX_{K}$, as in the diagram below. \[ \xymatrix{
            &  \cX_K \ar[r] \ar[d] & \cX \ar[d]^{\phi} \\ \overline{\xi} \colon
\Spec L  \ar[ur]^{\xi_0}  \ar[r] & \Spec K \ar[r] & \overline{\cX}, } \] Note
that $\cX_K$ is a $\gm$-gerbe over $K$. So $\xi_0$ (and hence, $\xi$) descends
to some intermediate subfield of $K \subset K_0 \subset L$ such that $\trdeg_K
K_0 \le \ed_k (\cX_K) = \cd (\cX_K)$, where the last equality is
Proposition~\ref{prop.ed-gerbe}(a). Let $\eta$ be the generic point of
$\overline{\cX}$. We know that the Brauer class of $\cX_{\eta}$ has index
dividing $n$ and exponent dividing $d$; see the proof of
Proposition~\ref{prop.generic-ed}. By Lemma~\ref{lem.index} the same is true of
the Brauer class of $\cX_{K}$. Therefore, by
Lemma~\ref{lem.brauer-severi.cd}(c), $\cd(\cX_K) \le \cd (\GL_n/\mu_d)$. In
summary, \[ \ed_{k}\xi \le \trdeg_k K_0 = \trdeg_k K + \trdeg_K K_0 \le
\ed_k(\phi(\xi)) + \cd(\cX_K), \] as claimed. \end{proof}

\medskip \begin{proof}[Proof of Theorem~\ref{thm.main}(b)] Let $\cX \eqdef
[\PP(A_{n, d})/\GL_n]$, $\overline{\cX} \eqdef [\PP(A_{n, d})/\PGL_n]$, and $\cY
\eqdef [A_{n, d}/\GL_n]$ and consider the following diagram of natural maps. \[
\xymatrix{   &  & \cY \ar[dll] \ar[ddll]^{\psi, \, \text{\tiny a $\mu_d$-gerbe}}
 \\ \cX \ar[d]_{\phi, \, \text{\tiny a $\gm$-gerbe}} &  & \\ \overline{\cX} &  &
}   \] In view of Proposition~\ref{prop.gerbes}, it suffices to show that for
every field extension $L/k$ and every $L$-point $\xi \colon \Spec L \to
\overline{\cX}$, 
\begin{equation} \label{e.genericity-base} \ed_{k} \phi(\xi)
\le \binom{n + d - 1}{d} - n^2 \, . \end{equation} 
Recall from Example~\ref{ex.amenable2} that under our 
assumptions on $n$ and $d$, we have that 
$\overline{\cX} = [\PP(A_{n, d})/\PGL_n]$ 
is amenable and is generically a scheme of dimension 
\[ \dim \PP(A_{n, d}) - \dim \PGL_n = \binom{n + d - 1}{d} - n^2 \, . \] 
If the automorphism group scheme $\underaut_{L}(\phi(\zeta))$ is
reductive then~\eqref{e.genericity-base} holds by the Genericity
Theorem~\ref{thm:genericity} applied to $\overline{\cX}$.

We may therefore assume that $\underaut_{L}(\phi(\zeta))$ 
is not reductive. Lift $\xi$ to some 
$\zeta\colon \spec L \arr \cY$. (This can be done because $\cY \to
\cX$ is a $\gm$-torsor.) Since the automorphism group scheme
$\underaut_{L}(\zeta)$ is contained in the preimage of
$\underaut_{L}(\phi(\zeta))$ under the natural projection map $\GL_n \to
\PGL_n$, we see that $\underaut_{L}(\zeta)$ is not reductive. Now
$\ed_{k}\phi(\xi) \le \ed_{k}\zeta$ and in view of
Theorem~\ref{thm:non-reductive} 
\[ \ed_{k}\zeta \le  \binom{n + d - 1}{d} - n^2 \, . \] 
This completes the proof of~\eqref{e.genericity-base} and thus of
Theorem~\ref{thm.main}(b). \end{proof}

\section{Small $n$ and $d$}
\label{sect.small}

In this section we compute $\ed_{k}\Forms_{n, d}$ and 
$\ed_{k}\Hypersurf_{n, d}$ in the cases not covered by 
Theorem~\ref{thm.main}, building on the results
of~\cite{bf1} and~\cite[Section 16]{ber}.

To handle the case where $n = 2$, we need the following 
variant of \cite[Lemma 16.1]{ber}. The proof is similar; 
we reproduce it here, with the necessary
modifications, for the sake of completeness.

\begin{lemma} \label{lem.binary} $\ed_{k}\Forms_{2, d} \le d-1$ and
$\ed_{k}\Hypersurf_{2, d} \le d-2$ for any $d \ge 3$. \end{lemma}

\begin{proof} Let $f(x_1, x_2) = a_0 x_1^d + a_1 x_1^{d-1} x_2 + \dots + a_d
x_2^d$ be a non-zero binary form of degree $d$ over a field $K/k$. We claim that
that $f$ is equivalent (up to a linear coordinate changes by elements of
$\GL_2(K)$) to a binary form with (i) $a_0 = 0$ or $a_1 = 0$ and  (ii) $a_{d-1}
= 0$ or $a_{d-1} = a_d$. In each case $f$ descends to the field
$k(a_0, \dots, a_d)$ and the hypersurface in $\mathbb{P}^1$ 
cut out by $f$ descends to the field $k(a_i/a_j \, | a_j \ne 0)$. 
If (i) and (ii) are satisfied then the transcendence degrees 
of these fields over $k$ are clearly $\le d-1$ and $d-2$, respectively. 
So, the lemma follows from the claim.

To prove the claim, we first reduce $f$ to a form satisfying (i).  If $a_0 = 0$,
we are done. If $a_0 \ne 0$, then performing the Tschirnhaus substitution 
\[ x_1 \mapsto x_1 - \frac{a_1}{d a_0} x_2 \, , \quad x_2 \mapsto x_2 \] 
we reduce $f$ to a binary form with $a_1 = 0$.

Now assume that $f$ satisfies (i). We want to further reduce it to a form
satisfying both (i) and (ii). If $a_{d-1} = 0$, we are done. 
If $a_{d-1} \ne 0$, rescale $x_1$ as follows 
\[ x_1 \mapsto \frac{a_d}{a_{d-1}} x_1 \, , \quad x_2
\mapsto x_2 \, ,  \] to reduce $f$ to a form satisfying (i) and $a_{d-1} = a_d$.
This completes the proof of the claim and the lemma. \end{proof}

\begin{proposition} \label{prop.hyper}For any $n \geq 1$ and $d \ge 2$ 
we have

\begin{enumeratea}

\item $\ed_{k}\Forms_{n, 1} = \ed_{k}\Hypersurf_{n, 1} = 0$,

\item $\ed_k \Forms_{1, d} = 1$ and  $\ed_{k} \Hypersurf_{1, d} = 0$, 

\item $\ed_{k}\Forms_{n, 2} = n$\quad and\quad $\ed_{k}\Hypersurf_{n, 2} = n-1$,

\item $\ed_{k}\Forms_{2, 3} = 2$\quad and\quad $\ed_{k}\Hypersurf_{2, 3} = 1$,

\item $\ed_{k}\Forms_{2, 4} = 3$\quad and\quad $\ed_{k}\Hypersurf_{2, 4} = 2$,

\item $\ed_{k}\Forms_{3, 3} = 4$\quad and\quad $\ed_{k}\Hypersurf_{3, 3} = 3$.

\end{enumeratea} \end{proposition}

\begin{proof} First we note that
\begin{equation} \label{e.+1}
\ed_{k}\Forms_{n, d} \le \ed_{k}\Hypersurf_{n, d} + 1 \, . 
\end{equation}
This is easy to see directly from the definition or, 
alternatively, as a special case of the Fiber Dimension
Theorem~\cite[Theorem 3.2(b)]{brv-jems}, applied 
to the representable morphism
of quotient stacks $ [A_{n, d}\setminus \{0\}/\GL_n] \to [\PP(A_{n, d})/\GL_n]$ of
relative dimension $1$.

\smallskip
(a) Any non-zero linear form $F(x_1, \dots, x_n)$ over any field
$K/k$ is equivalent to $x_1$.

\smallskip
(b) Degree $d$ forms $f_1(x) = ax^d$ and
$f_2(x) = bx^d$, over a field $K/k$ are equivalent if 
and only if $b = a c^d$ for some $c \in K^* = \GL_1(K)$.
The assertions of part (b) follow easily from this.

\smallskip (c) Any quadratic form $F(x_1, \dots, x_n)$ over $K/k$ can be
diagonalized and hence, is defined over an intermediate
field $k \subset K_0 \subset K$ such that
$\trdeg_k K_0 \le n$. This implies that
$\ed_{k}\Forms_{n, 2} \leq n$ and $\ed_{k}\Hypersurf_{n, 2} \leq n-1$.
The opposite inequalities follow from \cite[Proposition 16.2(b)]{ber}.

\smallskip (d) By Lemma~\ref{lem.binary}, $\ed_{k}\Forms_{2, 3} \le 2$ and
$\ed_{k}\Hypersurf_{2, 3} \le 1$, respectively. On the other hand,
by~\cite[Proposition 16.2(c)]{ber}, for the generic binary form $F_{\rm gen}$ in
three variables (as in~\eqref{e.generic-form})
and the hypersurface $H_{\rm gen}$ it cuts out in
$\PP^1$, we have $\ed_{k}F_{\rm gen} = 2$ and $\ed_{k}H_{\rm gen} = 1$.

\smallskip
Part (e) is proved in a similar manner, by combining Lemma~\ref{lem.binary} 
with~\cite[Proposition 16.2(d)]{ber}.

\smallskip (f) The identity $\ed_{k}\Hypersurf_{3, 3} = 3$ is the main result
of~\cite{bf1}. By~\eqref{e.+1}, $\ed_{k}\Forms_{3, 3} \le 4$.

In order to show that equality holds, it suffices to prove that the essential
dimension $\ged_{k}[X_{3, 3}/\GL_3]$ of the generic form $F_{\rm gen}$ of degree $3$ in $3$ variables 
 is at least $4$.  
By~\cite[Theorem A]{richardson2}
the $\GL_3$-action on $X_{3, 3}$ has a stabilizer in general position. Denote it
by $S_{3, 3}$, as in the proof of Lemma~\ref{lem.Gamma}. As we mentioned there
(and in Example~\ref{ex.amenable2}), $S_{3, 3}$ is a finite subgroup of $\GL_3$.
 Since the dimension of  $[X_{3, 3}/\GL_3]$ is $1$, by \cite[Lemma 15.4 and
Proposition 5.5(c)]{ber} $\ed_{k}F_{\rm gen} \ge \ed_{\overline{k}}(S_{3, 3}) +
1$, where $\ed_{k}S_{3, 3}$ denotes the essential dimension of the finite group
$S_{3, 3}$. (Note that in~\cite{ber} the symbol $\phi_{n, d}$ was used in place
of $F_{\rm gen}$.) Thus it suffices to show that $\ed_{\overline{k}}(S_{3, 3})
\ge 3$.

To get a better idea about the structure of $S_{3, 3}$, note that a form in
$X_{3, 3} (\overline{k})$ in general position is a scalar multiple of $x_1^3 +
x_2^3 + x_3^3  + 3a x_1 x_2 x_3$ for some $a \in \overline{k}$.  Hence, $S$
contains a non-abelian subgroup $H$ of order $27$, generated by diagonal
permutation matrices $\diag ( \zeta_1, \zeta_2, \zeta_3)$, where $\zeta_1$,
$\zeta_2$ and $\zeta_3$ are cube roots of unity satisfying $\zeta_1 \zeta_2
\zeta_3 = 1$, and the permutation matrices cyclically permuting $x_1$, $x_2$ and
$x_3$. Now \begin{equation} \label{e.S33} \ed_{\overline{k}}(S) \ge
\ed_{\overline{k}}(H) \ge 3 \, , \end{equation} where the second inequality is a
consequence of the Karpenko-Merkurjev theorem; see~\cite[Theorem
1.3]{meyer-reichstein}. This completes the proof of part (f). \end{proof}

\begin{remark} \hfil \begin{enumeratei}

\item Since $S$ is a finite subgroup of $\GL_3$, it has a natural faithful
$3$-dimensional representation. Hence, $\ed_{k}S_{3, 3} \le 3$, and both
inequalities in~\eqref{e.S33} are actually equalities.

\item The proof of part (e) shows that $\ed_{k}F_{\rm gen} = 4$, where $F_{\rm
gen}$ is the generic form of degree 3 in 3 variables, as
in~\eqref{e.generic-form}. This answers an open question posed after the
statement of Proposition 16.2 in~\cite{ber}.
\end{enumeratei} \end{remark}

\section{Essential dimension of singular curves}
\label{sect.curves}

In this section we use our new
Genericity Theorem~\eqref{thm:genericity} 
to strengthen~\cite[Theorem~7.3]{brv-jems} on the essential
dimension of the stack on (not necessarily smooth)
local complete intersection curves with finite
automorphism group presented in \cite[Theorem~7.3]{brv-jems}. 
Let us recall the set-up.  Denote by
$\fM_{g,n}$ the stack of all reduced $n$-pointed local complete intersection
curves of genus $g$, that is, the algebraic stack over $\spec k$ whose objects
over a $k$-scheme $T$ are finitely presented proper flat morphisms 
$\pi \colon C \arr T$, together with $n$ sections 
$s_1, \dots, s_n \colon T \arr C$,
where $C$ is an algebraic space, the geometric fibers of $\pi$
are connected reduced
local complete intersection curves of genus $g$,
and the image of each $s_i$ is contained in the smooth 
locus of $C \arr T$. (We do not require the images of the sections 
to be disjoint.)

The stack $\fM_{g,n}$ contains the stack $\cM_{g,n}$ of smooth $n$-pointed curves of
genus $g$ as an open substack (here the sections are supposed to be disjoint). By standard results in deformation theory, every
reduced local complete intersection curve is unobstructed, and is a limit of
smooth curves. Furthermore there is no obstruction to extending the sections,
since these map into the smooth locus. Therefore $\fM_{g,n}$ is smooth and
connected, and $\cM_{g,n}$ is dense in $\fM_{g,n}$. However, the stack
$\fM_{g,n}$ is very large (it is certainly not of finite type), and in fact it
is very easy to see that its essential dimension is infinite. Assume that we are in the stable range, i.e., $2g - 2 + n > 0$: then
in \cite{brv-jems}
we show that the essential dimension of the open substack
$\fM_{g,n}^{\mathrm{fin}}$ of $\fM_{g,n}$ of curves with finite automorphism
group equals the essential dimension of $\cM_{g,n}$.

Let $C$ be an object of $\fM_{g,n}$ defined over an algebraically closed field
$K$. We say that $C$ is \emph{reductive} if the automorphism group scheme
$\underaut_{K}C$ is reductive. The marked curve $C$ is \emph{not} reductive if
and only the smooth part $C_{\rm sm} \subseteq C$ contains a component that is
isomorphic to $\AA^{1}_{K}$ and contains no marked points. A reductive object of
$\fM_{g,n}$ is an object $C \arr S$, whose geometric fibers over $S$ are
reductive. It is not hard to see that the reductive objects form an open
substack $\fM^{\mathrm{red}}_{g,n}$ of $\fM_{g,n}$. Then our new genericity
theorem applies, and allows to conclude that the essential dimensions of
$\fM^{\mathrm{red}}_{g,n}$ and of $\cM_{g,n}$ are the same. From
\cite[Theorem~1.2]{brv-jems} we obtain the following.

\begin{theorem} \label{thm:ed-curves} If $2g-2+n > 0$ and 
the characteristic of $k$ is $0$, 
then \[ \ed_{k}\fM^{\mathrm{red}}_{g,n} = \begin{cases} 2         &
\text{if }(g,n)=(1,1), \\ 5         & \text{if }(g,n)=(2,0),\\ 3g-3 + n  &
\text{otherwise}. \end{cases} \] \end{theorem}

\bibliographystyle{amsalpha} 

\def\cprime{$'$}
\providecommand{\bysame}{\leavevmode\hbox to3em{\hrulefill}\thinspace}
\providecommand{\MR}{\relax\ifhmode\unskip\space\fi MR }
\providecommand{\MRhref}[2]{%
  \href{http://www.ams.org/mathscinet-getitem?mr=#1}{#2}
}
\providecommand{\href}[2]{#2}

\end{document}